\documentclass[11pt]{amsart}
\usepackage[table]{xcolor}
\usepackage{a4wide, amsmath, amsthm, mathabx, amscd, amsfonts, amssymb, caption,
	enumerate, epsfig, float, graphics, graphicx, hyperref, IEEEtrantools,
	mathrsfs, subcaption, textcomp, tikz, tikz-cd, verbatim, xcolor, xypic, ytableau}
\usepackage[T1]{fontenc}

\newtheorem{theorem}{Theorem}[section]
\newtheorem{proposition}[theorem]{Proposition}

\newtheorem{corollary}[theorem]{Corollary}
\newtheorem{lemma}[theorem]{Lemma}

\newtheorem{definition}[theorem]{Definition}
\newtheorem{construction}[theorem]{Construction}

\theoremstyle{definition}

\newtheorem{remark}[theorem]{Remark}

\makeatletter

\newcommand{\dashedrightarrow}[1][2pt]{%
  \settowidth{\@tempdima}{$\rightarrow$}\rightarrow
  \makebox[-\@tempdima]{\hskip-1.5ex\color{white}\rule[0.5ex]{#1}{1pt}}
  \phantom{\rightarrow}
}
\makeatother

\def\Spec{\operatorname{Spec}}
\def\Proj{\operatorname{Proj}}

\def\CC{\mathbb{C}}
\def\QQ{\mathbb{Q}}
\def\RR{\mathbb{R}}
\def\NN{\mathbb{N}}
\def\ZZ{\mathbb{Z}}
\def\PP{\mathbb{P}}

\def\dim{\operatorname{dim}}
\def\min{\operatorname{min}}

\def\cO{\mathcal{O}}

\def\Proj{\operatorname{Proj}}

\def\del{\partial}
\def\delb{\bar{\partial}}

\definecolor{darkgreen}{RGB}{0,153,0}
\definecolor{darkred}{RGB}{204,0,0}
\definecolor{darkblue}{RGB}{0,51,204}
\definecolor{red}{RGB}{242,43,29}
\hypersetup{colorlinks,linkcolor={darkblue},citecolor={darkblue},urlcolor={darkblue}}  

\begin{document}

\title{Liouville domains from Okounkov bodies}

\author{Marco Castronovo}
\address{Department of Mathematics, Columbia University}
\email{marco.castronovo@columbia.edu}

\begin{abstract}

Given a strictly concave rational PL function $\phi$ on a complete $n$-dimensional fan $\Sigma$,
we construct an exact symplectic structure of finite volume on $(\CC^\times)^n$ and a
family of functions $H_{\phi,\epsilon}$ called polyhedral Hamiltonians. We prove
that for each $\epsilon$ the one-periodic orbits of $H_{\phi,\epsilon}$ come in families
corresponding to finitely many primitive lattice points of $\Sigma$ and determine
their topology. When $\phi$ is negative on the rays of $\Sigma$, we show that the level
sets of polyhedral Hamiltonians are hypersurfaces of contact type.
As a byproduct, this construction provides a dynamical
model for the singularities of toric varieties obtained as degenerations
of Fano manifolds in any dimension via Okounkov bodies.

\end{abstract}

\maketitle
\thispagestyle{empty}

\section{Introduction}\label{SecIntro}

\subsection{Convexity in algebraic geometry}

To an ample divisor $D$ on a smooth $n$-dimensional complex projective variety
$X$ one can associate closed convex sets $\Delta_\nu(X,D)\subset\RR^n$ called
Okounkov bodies, one for each rank $n$ valuation $v:\CC(X)^\times \to \ZZ^n$ on the function
field of $X$. Roughly speaking $\Delta_\nu(X,D)$ measures how the sections $\Gamma(X,\cO(mD))$
grow as $m\to\infty$ from the point of view of $\nu$. When $X$ carries an action of a
reductive group $G$, these sets were used by Okounkov \cite{Ok} to
study multiplicities of irreducible representations of $G$. Even in the absence of symmetries,
one can use Okounkov bodies as developed by Lazarsfeld-Musta\c{t}\u{a} \cite{LM}
and Kaveh-Khovanskii \cite{KK} to construct degenerations of $(X,D)$ to polarized
toric varieties; see Anderson \cite{An}. In this sense Okounkov bodies can be
thought of as analogues of moment polytopes for varieties that do not carry natural
torus actions.

\subsection{Convexity in symplectic topology}

Endowing $\RR^{2n}$ with the standard symplectic structure $\omega_{std}$, one can ask what open subsets
$U\subset \RR^{2n}$ are symplectomorphic to each other. If $U$ is the interior of a
compact submanifold with boundary, a useful
invariant is the characteristic distribution $\operatorname{ker}({\omega_{std}}_{|\partial U})$,
and great efforts have been made to understand its general properties. Weinstein \cite{We78}
realized that convexity of $U$ implies the existence of closed integral curves (see also Krantz \cite[Propositions 3.1.6, 3.1.7]{Kr}
for the equivalence of geometric and function-theoretic convexity). This result was
generalized to star-shaped domains by Rabinowitz \cite{Ra}, and lead
to the notion of contact type hypersurface as appropriate generalization of
boundary of a convex domain in symplectic topology; see Weinstein \cite{We79}, Ekeland-Hofer
\cite{EH} and Eliashberg-Gromov \cite{EG}.

\subsection{Goal}

In this article, we establish a direct connection between convexity in algebraic
geometry and symplectic topology, by proving that Okounkov bodies $\Delta_v(X,D)$
give rise to families of contact type hypersurfaces in $X$, when it is
endowed with the symplectic structure $\omega_D$ induced by the ample divisor $D$. We also show
that the dynamics of these contact hypersurfaces is encoded to a great extent by the
combinatorics of the Okounkov body.

\subsection{Description of the constructions}

The results of this article rely on the following two constructions.
Suppose that the Okounkov body $P=\Delta_\nu(X,D)$ is a rational convex polytope;
this is true in many cases, for example when the valuation $\nu$
comes from the tropicalization of $X$ as in Kaveh-Manon \cite{KM}.
Denote $r\in\NN^+$ the minimum integer such that the
vertices of $rP$ are integral. Call $\Sigma$ the normal fan of $P$: its
rays $\rho\in\Sigma(1)$ are generated by the primitive inward-pointing normal vectors
$u_\rho\in\ZZ^n$ to the facets of $P$, and a set of rays spans a higher-dimensional
cone in $\Sigma$ if the intersection of the corresponding
facets is a lower-dimensional face of $P$. One can write
$$rP = P_{r\phi} = \{ \; m\in \RR^n \; : \; \langle m,u_\rho\rangle \geq r\phi(u_\rho) \; \textrm{for all} \; \rho\in\Sigma(1) \; \} \quad ,$$
where $\phi:\RR^n\to \RR$ is a strictly concave PL function whose domains of linearity are the
cones of $\Sigma$ and such that $\phi(\ZZ^n)\subseteq \QQ$; see e.g. \cite[Proposition 6.1.10, Theorem 6.1.14]{CLS} (where
concavity is called convexity).

\begin{construction}
Consider the complex torus $(\CC^\times)^n$ with coordinates $z_i$ for $1\leq i\leq n$,
and for any lattice point $m=(m^{(1)},\ldots, m^{(n)})\in rP\cap\ZZ^n$ call
$\chi^m(z_1,\ldots ,z_n)=z_1^{m^{(1)}}\cdots z_d^{m^{(n)}}$ the corresponding
character. The one-form
$$\theta_{r\phi}= -\frac{i}{2}\sum_{k=1}^d\frac{\sum_{m\in P_{r\phi}\cap\ZZ^d}m^{(k)}|\chi^m|^2}{\sum_{m\in P_{r\phi}\cap\ZZ^d}|\chi^m|^2}
	\left(\frac{dz_k}{z_k}-\frac{d\overline{z}_k}{\overline{z}_k} \right)$$
defines an exact symplectic structure $\omega_{r\phi}=d\theta_{r\phi}$ of finite volume on $(\CC^\times)^n$.
\end{construction}

The formula above has a simple geometric interpretation. Thinking the complex torus
as the maximal orbit of the toric variety $(\CC^\times)^n\subset X(\Sigma)$, the symplectic structure
$\omega_{r\phi}$ is induced by the ample Cartier divisor $D_{r\phi}=-\sum_{\rho\in\Sigma(1)}r\phi(u_\rho)D_\rho$
on it, where $D_\rho$ is the Zariski closure of the torus orbit corresponding
to the ray $\rho\in\Sigma(1)$; see Section \ref{SecPotential} for more details.
We verify in Proposition \ref{PropositionHamiltonianAction} that the action
of the real torus $(S^1)^n\subset (\CC^\times)^n$ is Hamiltonian with respect to
$\omega_{r\phi}$, and that the function $\mu_{r\phi}:(\CC^\times)^n\to\RR^n$ given by

$$\mu_{r\phi}=\frac{1}{\sum_{m\in P_{r\phi}\cap\ZZ^n}|\chi^m|^2}
\sum_{m\in P_{r\phi}\cap\ZZ^n}|\chi^m|^2m $$

is a moment map in the sense of symplectic topology; see e.g. \cite[Section 4.2]{Fu} for a
discussion of the map $\mu_{r\phi}$ from the point of view of toric geometry. Note that versions of this
statement have appeared in the literature with various smoothness
assumptions on $X(\Sigma)$, which do not hold in our case since toric varieties
of associated Okounkov bodies can be very singular.

\begin{figure}[H]
  \centering
    \begin{subfigure}[b]{0.2\textwidth}
        \includegraphics[width=\textwidth]{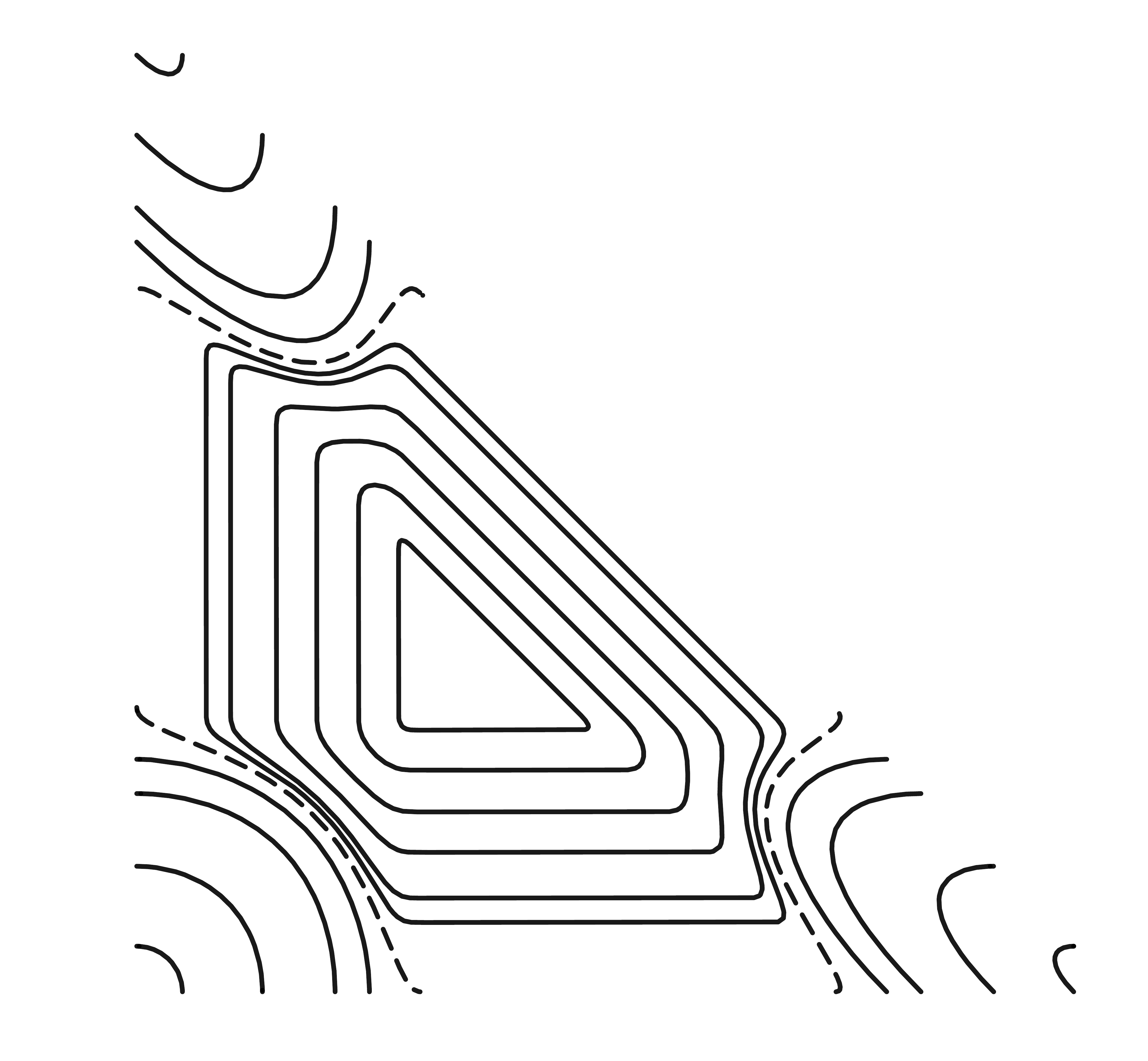}
    \end{subfigure}
    \begin{subfigure}[b]{0.2\textwidth}
        \includegraphics[width=\textwidth]{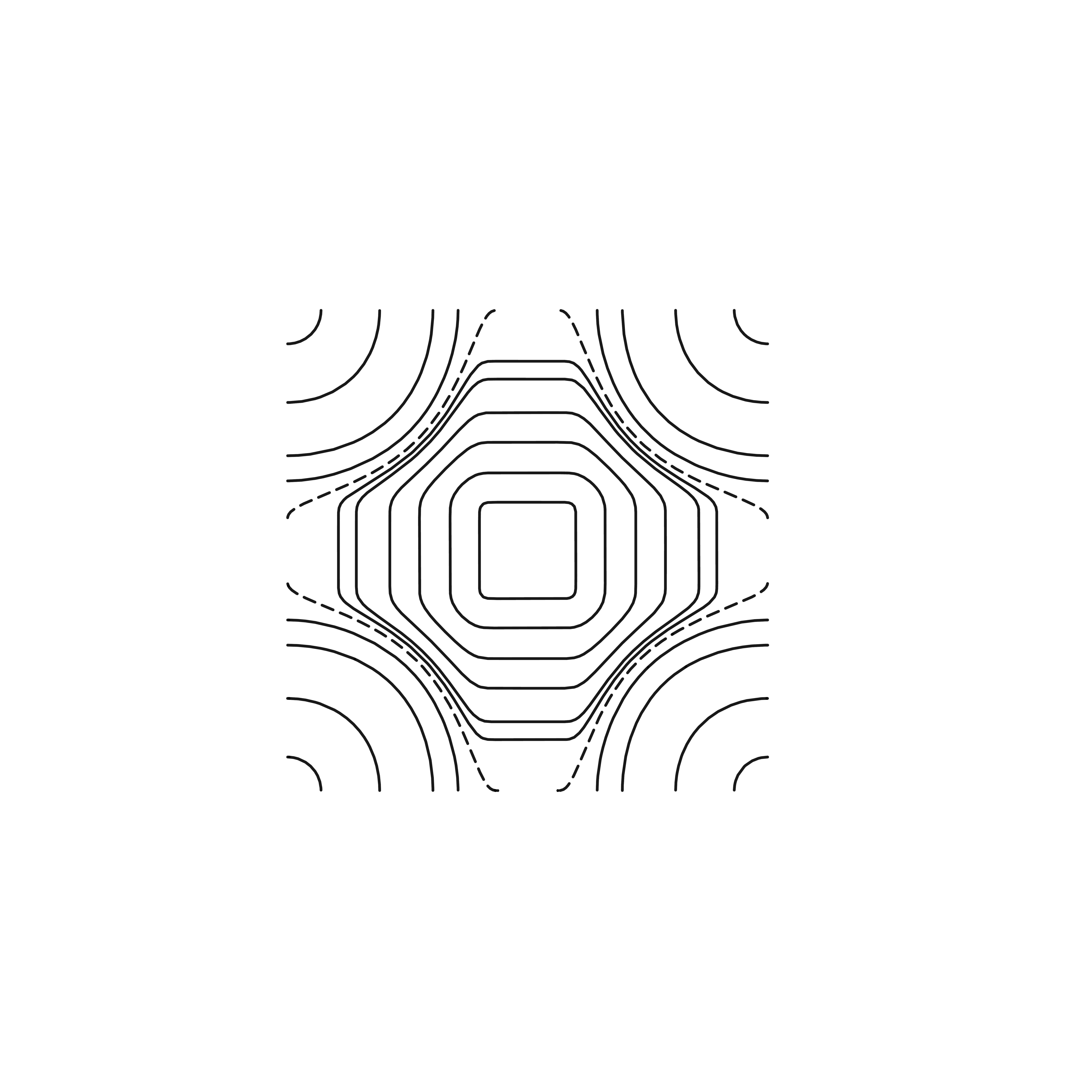}
    \end{subfigure}
    \begin{subfigure}[b]{0.2\textwidth}
        \includegraphics[width=\textwidth]{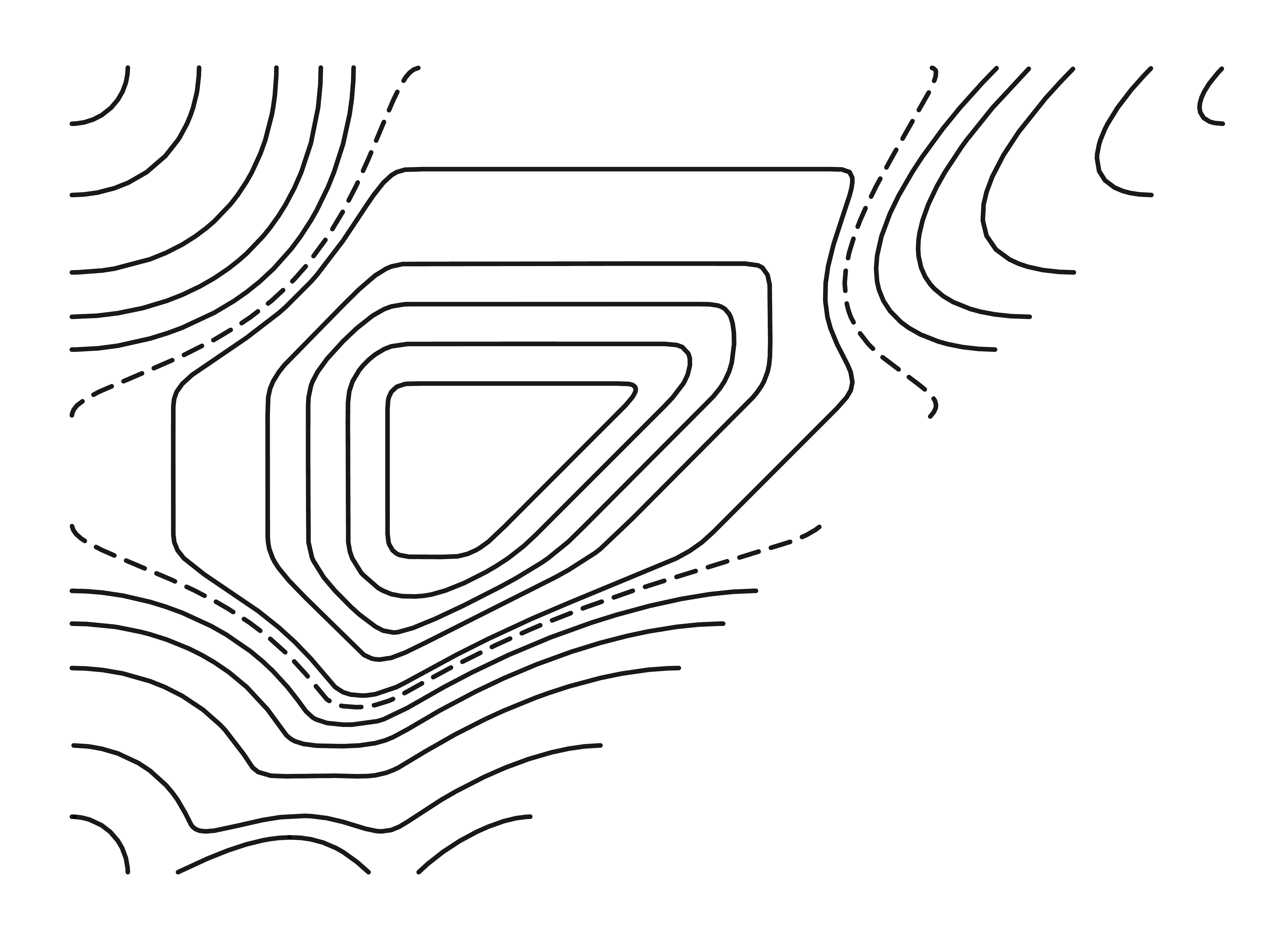}
    \end{subfigure}
    \begin{subfigure}[b]{0.2\textwidth}
        \includegraphics[width=\textwidth]{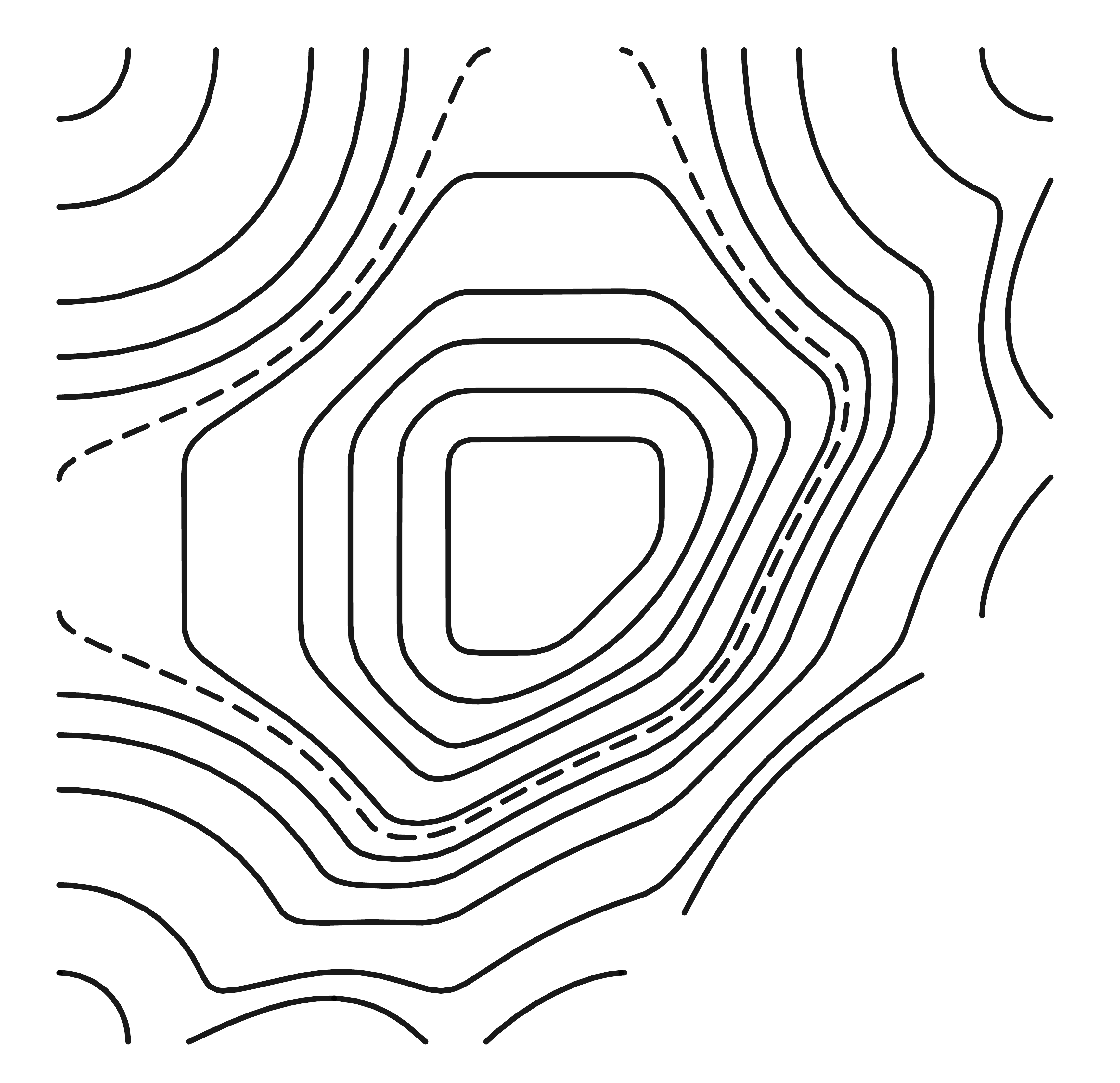}
    \end{subfigure}
    
    \begin{subfigure}[b]{0.2\textwidth}
        \includegraphics[width=\textwidth]{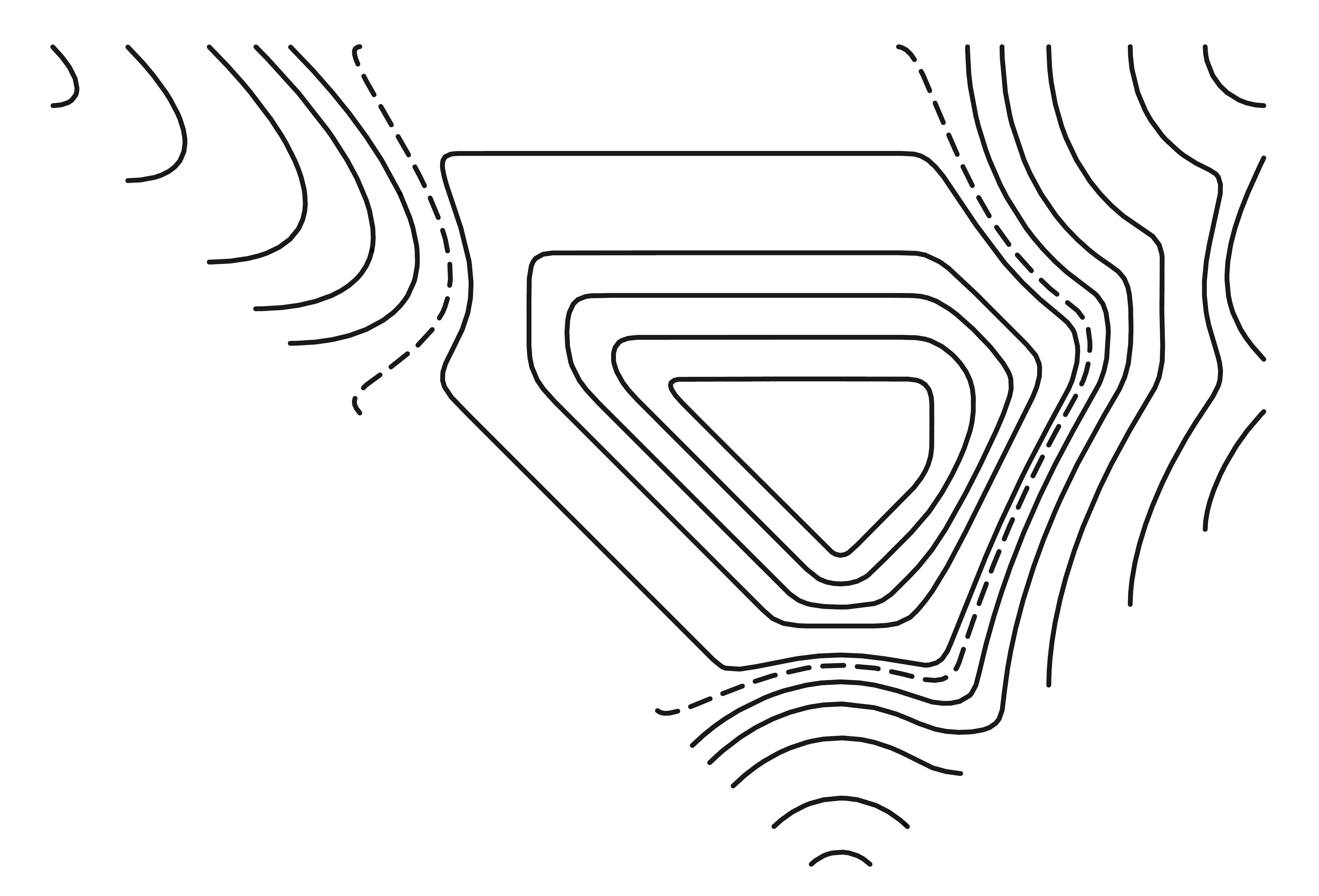}
    \end{subfigure}
    \begin{subfigure}[b]{0.2\textwidth}
        \includegraphics[width=\textwidth]{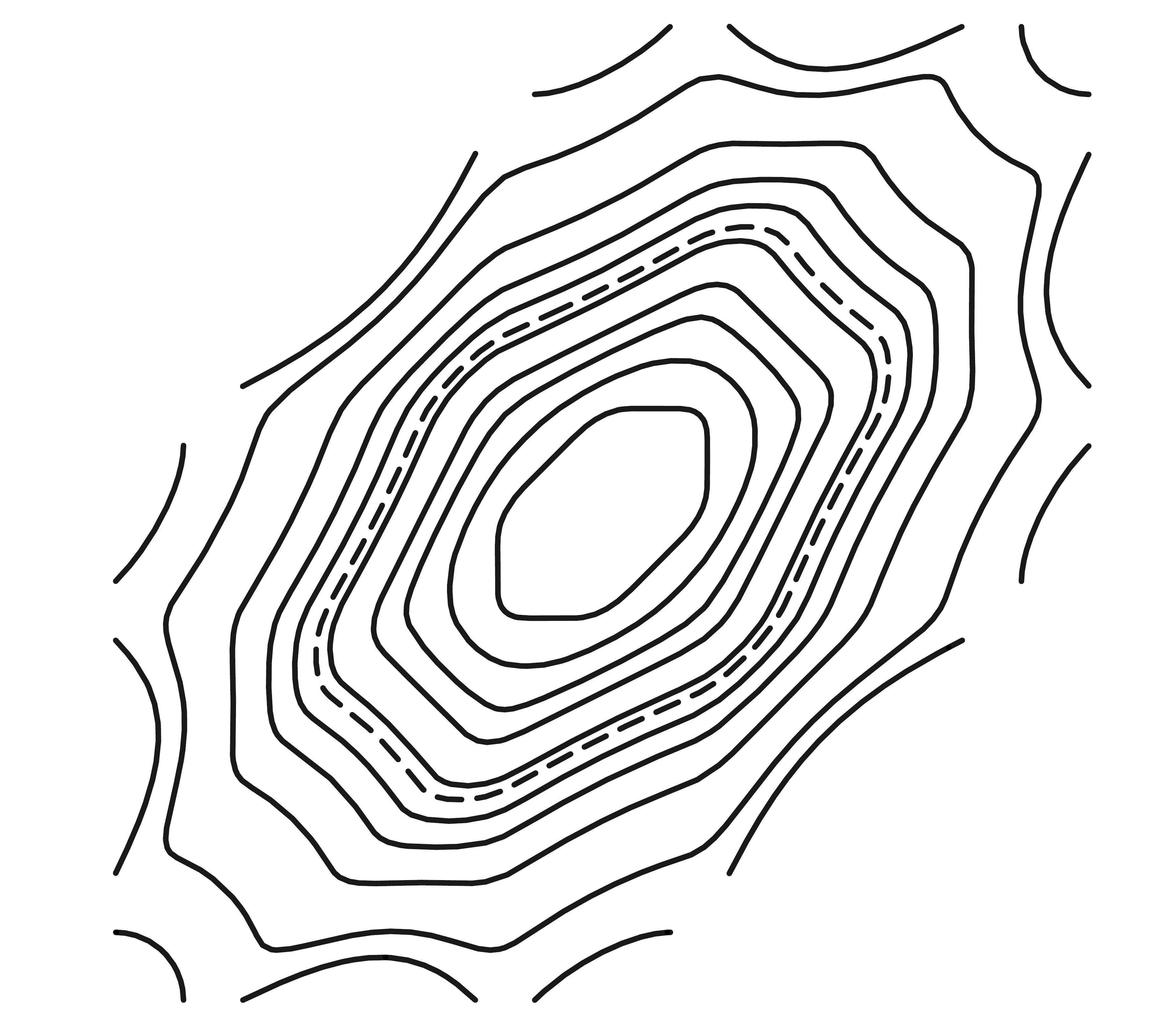}
    \end{subfigure}
    \begin{subfigure}[b]{0.2\textwidth}
        \includegraphics[width=\textwidth]{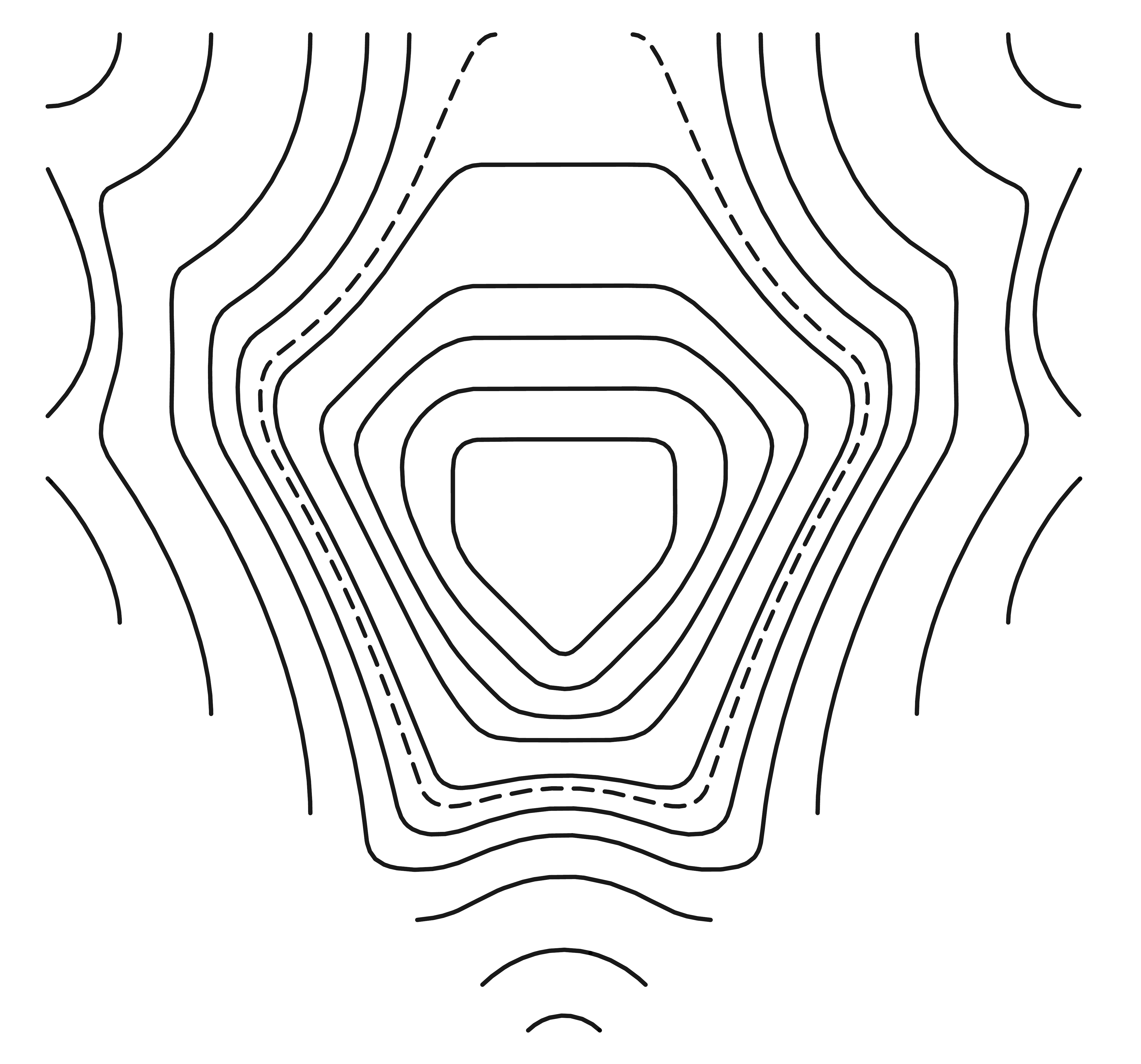}
    \end{subfigure}
    \begin{subfigure}[b]{0.2\textwidth}
        \includegraphics[width=\textwidth]{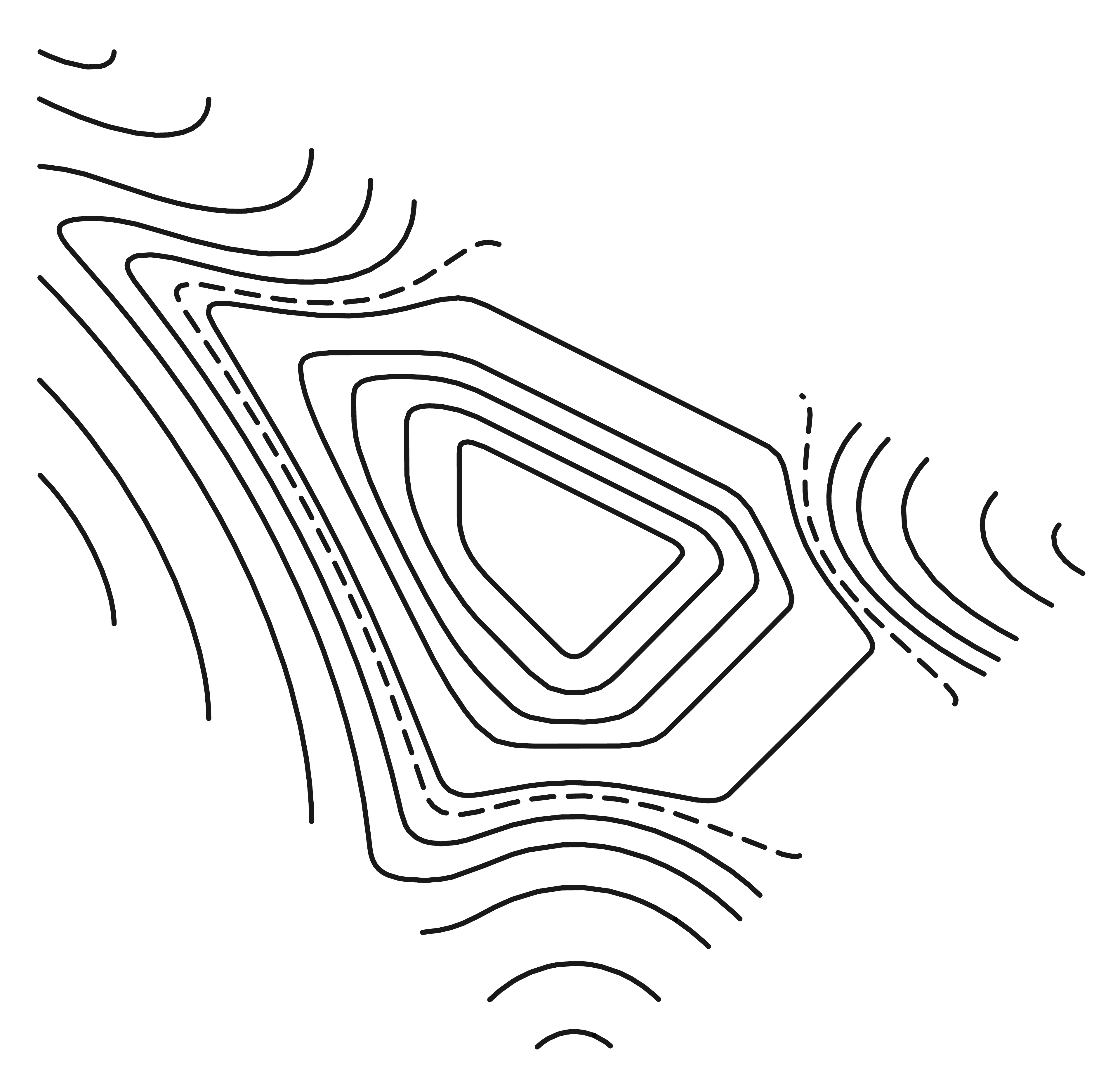}
    \end{subfigure}

    \caption{Level sets of the smoothing function $h_{r\phi,\epsilon}$ for some two-dimensional
    polytopes $P_{r\phi}$. The level set with value one is dashed.}
    \label{FigLevelSets}
\end{figure}

The second construction is a smoothing procedure for convex polytopes of arbitrary
dimension.

\begin{construction}
For any $\epsilon = (\epsilon_\rho )_{\rho\in\Sigma(1)}\in\RR^{\Sigma(1)}_{> 0}$ consider
the function $h_{r\phi,\epsilon}:\operatorname{int}(P_{r\phi})\to \RR $ given by
$$h_{r\phi,\epsilon}(m)=\sum_{\rho\in\Sigma(1)}q_{\epsilon_\rho}(r_\rho(m)) \quad ,$$
where $r_\rho(m)=\langle m,u_\rho\rangle -r\phi(u_\rho)$ measures the distance from the
facet of $P_{r\phi}$ corresponding to $\rho\in\Sigma(1)$ and $q_{\epsilon_\rho}$ is a bump
function. The level sets $h^{-1}_{r\phi,\epsilon}(\delta)$ with $\delta\in (0,1)$ are smooth hypersurfaces
homeomorphic to $S^{n-1}$.
\end{construction}

The level sets of the smoothing function $h_{r\phi}$ can be thought of as smooth approximations
of the polyhedral boundary $\partial P_{r\phi}$ depending on $\epsilon$; see Figure
\ref{FigLevelSets} and Proposition \ref{PropLevelSpheres}. Composing with the moment map one gets a family of functions $H_{r\phi,\epsilon}=h_{r\phi,\epsilon}\circ\mu_{r\phi}$
on $(\CC^\times)^n$ that we call polyhedral Hamiltonians. The distance functions also lift
to functions $R_{\rho}=r_\rho\circ\mu_{r\phi}$ on $(\CC^\times)^n$, which intuitively
measure the distance from the component $D_\rho$ of the torus-invariant divisor $D_{r\phi}$
in the compactification $(\CC^\times)^n\subset X(\Sigma)$. The Hamiltonian vector field
$X_{R_\rho}$ corresponding to $R_\rho$ coincides with the infinitesimal action $X_{u_\rho}$
of the one-parameter subgroup $\lambda_{u_{\rho}}:\CC^\times \to (\CC^\times)^n$ given
by $\lambda_{u_{\rho}}(t)=(t^{u_\rho^{(1)}},\ldots ,t^{u_\rho^{(n)}})$. This smoothing procedure mimics
McLean's construction of the link of a simple normal crossing divisor \cite{ML16}. The
relevant divisor in this case is $D_{r\phi}$ on $X(\Sigma)$, which is however not simple normal
crossing due to the singularities of $X(\Sigma)$. We bypass
this issue by exploiting the torus symmetry, and use the moment map to
reduce the construction to the combinatorics of $P_{r\phi}$ rather than using
regularizing tubular neighborhoods for the divisor $D_{r\phi}$. Note that in principle one
could use resolutions of singularities to make $D_{r\phi}$ simple normal crossing in a smooth
birational model for $X(\Sigma)$, where McLean's techniques can be applied. However,
this would make the results of this article less explicit, and further introduce subtle questions
regarding the dependence on the choice of resolution.

\subsection{Results: contact type hypersurfaces}

For any $\epsilon\in\RR_{>0}^{\Sigma(1)}$ and $\delta\in (0,\infty)$ define
$$W_{\epsilon,\delta}(r\phi) = \{ \; z\in (\CC^\times)^n \; : \;
H_{\epsilon, r\phi}(z)\leq \delta \; \} \quad ;$$
this is a submanifold with boundary of $(\CC^\times)^n$ with a Lagrangian torus fibration
given by the moment map $\mu_{r\phi}$.

\begin{theorem} (Theorem \ref{ThmContactType})
If $\phi(u_\rho)<0$ for all $\rho\in\Sigma(1)$, then $\theta_{r\phi}$ restricts to a contact form on the hypersurface
$\partial W_{\epsilon,\delta}(r\phi)=H_{\epsilon, r\phi}^{-1}(\delta)\subset (\CC^\times)^n$
for all $\delta\in (0,\infty)$ and $\epsilon\in \RR_{>0}^{\Sigma(1)}$ such that $\epsilon_\rho < -r\phi(u_\rho)$ for all $\rho\in\Sigma(1)$ .
\end{theorem}

When $\delta\in (0,1)$ the boundary $\partial W_{\epsilon,\delta}(r\phi)=
H_{\epsilon, r\phi}^{-1}(\delta)$ is homeomorphic to $S^{n-1}\times(S^1)^n$,
and $W_{\epsilon,\delta}(r\phi)$ is a Liouville domain; for $\delta \geq 1$
the boundary can be non-compact. The construction of these domains from
the Okounkov body $P_{r\phi}=\Delta_\nu(X,D)$ is carried out in
a complex torus, but they also have symplectic embeddings in the original projective
manifold $X$, once endowed with the symplectic structure $\omega_D$ induced
by the ample divisor $D$. This holds because, by a result of Harada-Kaveh \cite{HK},
the projective manifold $X$ contains a complex analytic open set symplectomorphic
to $((\CC^\times)^n,\omega_{r\phi})$. The assumption that $\phi(u_\rho)<0$ for all $\rho\in\Sigma(1)$
is equivalent to saying that the Cartier divisor $D_{r\phi}=-\sum_{\rho\in\Sigma(1)}r\phi(u_\rho)D_\rho$
on the toric compactification $X(\Sigma)$ is strictly effective. The proof of the theorem
reduces to
establishing the formula $\theta_{r\phi}(X_{u_\rho})=\langle \mu_{r\phi}, u_\rho\rangle$,
which is done in two steps. First one proves a general distortion formula of the type
$\theta_{r\phi}(X_{u_\rho})=\langle \mu_{r\phi}, u_\rho\rangle + C_\rho$ for some error
term $C_\rho\in\RR$ depending only on $\rho\in\Sigma(1)$ (Lemma \ref{LemmaDistortionFormula});
this holds thanks to the fact that the action of $(S^1)^n\subset (\CC^\times)^n$ is not
only Hamiltonian, but also strictly exact (Proposition \ref{PropositionHamiltonianAction}).
Then one verifies that $C_\rho=0$ for all $\rho\in\Sigma(1)$ (Lemma \ref{LemmaDistortionConstant});
this calculation relies on what we call the wrapping-averaging formula (Proposition \ref{PropWrappingAverageFormula}):
$$\int_{|t|=a}\lambda_{u_\rho}^*\theta_{r\phi} = 2\pi \langle \overline{m}_{\rho,r\phi}(a),u_\rho\rangle \quad .$$
This formula relates the wrapping numbers of one-parameter subgroups $\lambda_{u_{\rho}}:\CC^\times\to (\CC^\times)^n$
as measured by $\theta_{r\phi}$ to certain weighted averages of lattice points in the polytope $P_{r\phi}$:
$$\overline{m}_{\rho,r\phi}(a)=
\frac{\sum_{m\in P_{r\phi}\cap\ZZ^d}a^{2\langle m,u_\rho\rangle}m}{\sum_{m\in P_{r\phi}\cap\ZZ^d}a^{2\langle m,u_\rho\rangle}} \quad .$$
Besides being used in the proof of the theorem above, this formula has two more consequences.
One is that it can be used to show (Corollary \ref{CorInfinitesimalWrapping}) that the infinitesimal wrapping numbers
as $a\to 0$ recover the coefficients of the Cartier divisor $D_{r\phi}$; this is expected
from the simple normal crossing case, but the familiar techniques do not apply
here due to the singularities of $X(\Sigma)$. A second consequence is that the Lagrangian
torus given by the moment fiber $\mu_{r\phi}^{-1}(0)\subset W_{\epsilon,\delta}(r\phi)$ is
exact (Corollary \ref{CorExactTorus}).

\subsection{Results: families of periodic orbits}

Since the level sets $H_{r\phi,\epsilon}^{-1}(\delta)=\partial W_{\epsilon,\delta}(r\phi)\subset (\CC^\times)^n$ are
hypersurfaces of contact type by Theorem \ref{ThmContactType}, one expects the vector field $X_{H_{r\phi,\epsilon}}$
to have periodic orbits on them; see the initial discussion on symplectic convexity. In fact, the torus symmetry allows to produce large families
of such orbits. We show that these can be encoded by the lattice points of
the normal fan $\Sigma$ of the polytope $P_{r\phi}$ in the following sense.
If $\sigma\in\Sigma$ is a cone and $v\in\operatorname{int}(\sigma)\cap\ZZ^d$
is a lattice point in its relative interior, one can write
$$v = \sum_{\rho\in\sigma(1)}d_\rho u_\rho \quad \textrm{for some} \; d=(d_\rho)_{\rho\in\sigma(1)}\in\QQ_{>0}^{\sigma(1)} \quad ;$$
call $c_\sigma(d)$ the right hand side of this equation.
Here $\sigma(1)$ is the set of rays of the cone $\sigma$ and $u_\rho$ is the
primitive generator of the ray $\rho\in\sigma(1)$. Note that even though the vectors
$u_\rho$ and $v$ are integral, one might have $d_\rho\notin \ZZ_{>0}$; moreover,
this expression as linear combination is in general not unique. This happens because
we make no assumption of smoothness on $X(\Sigma)$, hence the cones $\sigma\in\Sigma$
can be non-smooth or even non-simplicial. 
We show that for each $d$ such that $c_\sigma(d)\in\operatorname{int}(\sigma)\cap\ZZ^n$ one
has a family of periodic orbits
$B^\epsilon_\sigma(d)\subset (\CC^\times)^n$ of the vector field $X_{H_{r\phi,\epsilon}}$.
This is a submanifold whose topology only depends on $\sigma$, and can be
explicitly described. It may happen that $B^\epsilon_\sigma(d)=\emptyset$, and
we call dynamical support $\operatorname{DS}_\sigma(r\phi,\epsilon)\subset\operatorname{int}(\sigma)\cap\ZZ^n$
the set of lattice points $v$ such that $v=c_\sigma(d)$ for some $d$ with $B^\epsilon_\sigma(d)\neq \emptyset$;
this set depends on $\epsilon$ in general.

\begin{theorem} (Theorem \ref{ThmTopologyOfFamilies})
If $H_{r\phi,\epsilon}$ is the polyhedral Hamiltonian associated to the PL function
$r\phi$ on $\Sigma$ with smoothing parameter $\epsilon$, then for any $\sigma\in\Sigma$ the following facts hold:
\begin{enumerate}
	\item the dynamical support $\operatorname{DS}_\sigma(r\phi,\epsilon)$ is finite ;
	\item for any $d\in c_\sigma^{-1}(\operatorname{DS}_\sigma(r\phi,\epsilon))$ the family $B^\epsilon_\sigma(d)$ is a smooth manifold
		diffeomorphic to a disjoint union of thickened tori $\operatorname{int}(D^{n-\dim\sigma})\times (S^1)^n$ . 
\end{enumerate}
\end{theorem}
 
The proof of this theorem consists in observing that $(\CC^\times)^n$ decomposes into locally
closed sets $S^\epsilon_\sigma$ that are invariant under the action of $(S^1)^n\subset (\CC^\times)^n$,
and such that for $z\in S^\epsilon_\sigma$ one has
$$X_{H_{r\phi,\epsilon}}(z)=\sum_{\rho\in\sigma(1)}q'_{\epsilon_\rho}(r_\rho(\mu_{r\phi}(z)))X_{u_\rho}(z) \quad ;$$
see Proposition \ref{PropSmoothingVectorField}. From this formula one sees that the dynamics
of $X_{H_{r\phi,\epsilon}}$ in the set $S^\epsilon_\sigma$ is that of a linear flow on each
moment fiber. Although the flow depends in general on which moment fiber one looks at,
all the fibers $\mu_{r\phi}^{-1}(m)\subset S^\epsilon_\sigma$ such that
$q'_{\epsilon_{\rho}}(r_\rho(m))=-d_\rho$ for all $\rho\in\sigma(1)$ have the same periodic flow,
and this constraint is satisfied on a disjoint union of open sets in $\operatorname{int}(P_{r\phi})$
diffeomorphic to balls whose dimension depends only on $\sigma$, by construction of the bump functions $q_{\epsilon_\rho}$.
Intuitively, the parameter $d$ prescribes slopes for the linear flow on each moment fiber, and
Proposition \ref{PropPeriods} computes the period of an orbit $\gamma\subset B^\epsilon_\sigma(d)$
to be
$$T(\gamma)=\frac{1}{|\operatorname{gcd}(\langle c_\sigma(d),e_k\rangle : k=1,\ldots ,n)|} \quad .$$
In particular, $\gamma$ has period one if and only if $c_\sigma(d)$ is a primitive
lattice point in the relative interior of the cone $\sigma$.

\subsection{Relation to HMS for Fano manifolds}

A large class of examples to which Theorem \ref{ThmContactType}
applies arises from the case where the projective manifold $X$ is Fano, and $D\in |K_X^{-1}|$
is an anti-canonical divisor. In this case the Okounkov bodies $P=\Delta_v(X,D)$
are polar dual to Fano polytopes in the sense of Akhtar-Coates-Galkin-Kasprzyk \cite{ACGK}, and
the associated toric varieties $X(\Sigma)$ are $\QQ$-Fano. This means that, in the constructions above,
one can take $\phi$ to be the support function of the toric $\QQ$-Cartier
anti-canonical divisor of $X(\Sigma)$: $\phi(u_\rho)=-1$
for all $\rho\in\Sigma(1)$; in particular, the assumption $\phi(u_\rho)<0$ of the theorem is
satisfied. In Homological Mirror Symmetry (HMS), the pair $(X,D)$ is expected to
have a partner Landau-Ginzburg model $(X^\vee,W)$ consisting of a complex variety
with a regular function $W\in\cO(X^\vee)$ called potential. The work of Tonkonog \cite{To}
suggests that Lagrangians $L\subset X$ that are monotone with respect to $\omega_D$
and become exact in a Liouville subdomain of $X$ should correspond to subschemes
$U\subset X^\vee$, with $W_{|U}$ being a generating function of rigid
pseudo-holomorphic curves. These curves are half-cylinders that
have boundary on $L$, and are obtained by neck-stretching from global pseudo-holomorphic
disks. Previous work of the author \cite{Ca20, Ca21} verifies a similar correspondence
between certain Lagrangian tori in complex Grassmannians and cluster charts
of a mirror Landau-Ginzburg model proposed by Rietsch \cite{Ri} (see also Marsh-Rietsch \cite{MR}); the proof relies on
toric degenerations induced by Okounkov bodies for the Grassmannian
previously studied by Rietsch-Williams \cite{RW}. The results of this article
could be useful where, in contrast with the case of Grassmannians, a candidate
Landau-Ginzburg model for $(X,D)$ is not already known. At the level of speculation,
one can imagine to construct a candidate Landau-Ginzburg model $(X^\vee,W)$ as gluing of algebraic torus charts,
one for each Okounkov body $\Delta_v(X,D)$, with
$W$ defined chart by chart as generating function of rigid pseudo-holomorphic half-cylinders in the completion of
the corresponding Liouville domains constructed in this article. The reader
is referred to the Gross-Siebert program \cite{GS18, GS19, GS21} for a	
construction of Landau-Ginzburg models based on (closed) logarithmic Gromov-Witten theory.

\subsection{Future directions}

We describe now some future directions of research that stem from this article
and we hope to explore in the future. In Lagrangian Floer theory, a result of Nishinou-Nohara-Ueda \cite{NNU}
computes the disk potential of Lagrangian tori obtained from degenerations of projective
manifolds, provided that the limit of the degeneration is a toric variety with a
small resolution of singularities. Theorem \ref{ThmContactType} constructs Liouville
domains in which these tori become exact, and one can use neck-stretching along the
contact boundary in the sense of symplectic field theory \cite{EGH} to constraint the global pseudo-holomorphic disks
that the Lagrangian tori can bound. Thanks to Theorem \ref{ThmTopologyOfFamilies}, the constraints
would be specific to the nature of the singularities of the degeneration at hand.
In a different direction, in the study of symplectic capacities explicit calculations for Liouville domains with torus symmetry
are often possible; see for example Gutt-Hutchings \cite{GH} and Siegel \cite{Si}. Similar
calculations for the Liouville domains introduced in this article would yield, by monotonicity
of capacities, lower bounds for symplectic capacities of projective manifolds that admit
toric degenerations; see Kaveh \cite{Ka} for similar results on the Gromov width. Finally,
there is a long history of results relating the birational geometry of varieties with their
symplectic topology; see McLean \cite{ML20} for a recent breakthrough. It is natural to
ask if the existence of special resolutions of singularities of the pair $(X(\Sigma),D_{r\phi})$ is
reflected in some algebraic property of symplectic cohomology of the corresponding
Liouville domains $W_{\epsilon,\delta}(r\phi)$; see Evans-Lekili \cite{EL} for some results
relating small resolutions and symplectic cohomology in the context of Du Val singularities. 

\vspace{0.3cm}

\textbf{Acknowledgements} I thank Chris Woodward for introducing me to toric degenerations,
and Lev Borisov for being the first to mention Okounkov bodies. This work benefited from
conversations with Mohammed Abouzaid and Mark McLean. I also thank Francesco Lin for
a discussion related to Remark \ref{RmkExoticSpheres}.

\section{Combinatorics of fans}\label{SecFans}

In this subsection, we recall some basic facts and notations about polyhedral fans;
see e.g. \cite{CLS, Fu} as general reference.

\begin{definition}\label{DefCone}
A polyhedral cone $\sigma\subset\RR^n$ is the convex hull of finitely many rays
starting from the origin. A cone is strongly convex (or sharp) if it contains no line,
and is rational if each ray $\rho$ has a primitive generator $u_\rho\in\ZZ^n$.
\end{definition}

\begin{definition}\label{DefFan}
A set $\Sigma$ of cones in $\RR^n$ is a fan if:
\begin{enumerate}
	\item every cone $\sigma\in\Sigma$ is polyhedral, strongly convex and rational ;
	\item if $\tau\subseteq \sigma$ is a face and $\sigma\in\Sigma$ then $\tau\in\Sigma$ ;
	\item if $\sigma_1,\sigma_2\in\Sigma$ then $\sigma_1\cap\sigma_2\in\Sigma$ .
\end{enumerate}
The support $|\Sigma|$ of a fan is the union of its cones, and $\Sigma(k)$ denotes the
set of its $k$-dimensional cones for each $0\leq k\leq n$. A fan is complete if $|\Sigma|=\RR^n$.
\end{definition}

To each complete fan in $\RR^n$, one associates a proper normal variety $X(\Sigma)$
with an action of $(\CC^\times)^n$, whose orbits are in bijection with $\Sigma$. The orbit
$\cO(\sigma)\subset X(\Sigma)$ has $\operatorname{codim}\cO(\sigma)=\operatorname{dim}\sigma$,
and $\sigma_1\subseteq \sigma_2$ if and only if $\overline{\cO(\sigma_1)}\cap\cO(\sigma_2)\neq\emptyset$
in the Zariski topology.

\begin{definition}\label{DefSpecialOrbits}
If $\Sigma$ is a complete fan, call $\cO(\{0\})=(\CC^\times)^n$
the maximal orbit, and $X(\Sigma)\setminus \cO(\{0\})=D_\Sigma$ the toric anti-canonical
divisor. Also call $\overline{\cO(\rho)}=D_\rho$ the prime divisor associated
to the ray $\rho\in\Sigma(1)$.
\end{definition}

\begin{definition}\label{DefPLFunction}
If $\Sigma$ is a complete fan of cones in $\RR^n$, a PL function on $\Sigma$ is
a continuous function $\phi : \RR^n\to \RR$ such that $\phi_{|\sigma}$ is linear for all
$\sigma\in\Sigma$. On a maximal cone $\sigma\in\Sigma(n)$, any such function is of the
form $\phi=\langle m_\sigma,-\rangle$ for a unique $m_\sigma\in\RR^n$, and $\phi$ is
called integral (resp. rational) when $m_\sigma\in\ZZ^n$ (resp. $m_\sigma\in\QQ^n$) for
all $\sigma\in\Sigma(n)$.
\end{definition}

Any divisor (resp. $\QQ$-divisor) on $X(\Sigma)$ is linearly equivalent to a torus-invariant
one, and the latter are of the form $\sum_{\rho\in\Sigma(1)}a_\rho D_\rho$ for some $a_\rho\in\ZZ$
(resp. $a_\rho\in\QQ$). The Cartier (resp. $\QQ$-Cartier) divisors are precisely those
for which there is an integral (resp. rational) PL function $\phi$ on $\Sigma$
such that $a_\rho=-\phi(u_\rho)$ for all $\rho\in\Sigma(1)$, in which case we write
$D_\phi = -\sum_{\rho\in\Sigma(1)}\phi(u_\rho)D_\rho$.

\begin{definition}\label{DefConcavity}
A function $\phi:\RR^n\to\RR$ that is PL on a complete fan $\Sigma$ is called
concave if $\phi(x)=\operatorname{min}_{\sigma\in\Sigma(n)}\langle m_\sigma,x\rangle$,
and strictly concave when for all $\sigma\in\Sigma(n)$ one has $\phi(x)=\langle m_\sigma,x\rangle$
if and only if $x\in\sigma$.
\end{definition}

The divisor $D_\phi$ on $X(\Sigma)$ is basepoint-free if and only if $\phi$ is concave,
and ample if and only if $\phi$ is strictly concave.

\begin{definition}\label{DefSectionPolytope}
If $\Sigma$ is a complete fan of cones in $\RR^n$ and $\phi$ is a PL function on $\Sigma$,
call section polytope of $\phi$ the set
$$P_\phi = \{ \; m\in\RR^n \; : \; \langle m,u_\rho\rangle \geq \phi (u_\rho) \; \textrm{for all}
\; \rho\in\Sigma(1) \; \} \quad .$$
\end{definition}

Writing a lattice point of the section polytope $m\in P_\phi\cap \ZZ^n$ as $m=(m^{(1)},\ldots ,m^{(n)})$,
one has an associated character $\chi^m(z_1,\ldots ,z_n)=z_1^{m^{(1)}}\cdots z_n^{m^{(n)}}$ of
the complex torus $(\CC^\times)^n$. If $r\in\NN^+$ is such that $rD_{\phi}=D_{r\phi}$ is
Cartier, the characters $\chi^m$ with $m\in P_\phi\cap \ZZ^n$ form
a basis of the space of sections $\Gamma(X(\Sigma),\cO(D_{r\phi}))$.

\section{Character sums and K\"{a}hler potential}\label{SecPotential}

\subsection{The Fubini-Study potential}

Denote $\PP^d=\Proj\CC[x_0,\ldots ,x_d]$ the $d$-dimensional complex projective
space, and $U_k=\{ \; [x_0:\cdots :x_d] \; : \; x_k \neq 0 \}$ for $0\leq k\leq d$
the $d+1$ open sets of a holomorphic atlas with charts
$$\psi_k: U_k\to \CC^d \quad , \quad [x_0:\dots :x_d] \mapsto
	\left( \frac{x_0}{x_k}, \ldots ,\widehat{\frac{x_k}{x_k}}, \ldots ,\frac{x_d}{x_k} \right) \quad .$$
Thinking the $d$-dimensional affine space as $\CC^d=\Spec\CC[z_1,\ldots ,z_d]$,
the function
$$\rho : \CC^d \to \RR \quad , \quad \rho (z_1,\ldots ,z_d)=\log (1+|z_1|^2+\ldots +|z_d|^2)$$
is plurisubharmonic, and hence defines a symplectic form $\omega_{FS}^{loc}=i\del\delb\rho$
on $\CC^d$ which is compatible with the standard complex structure $J_{\CC^d}$.
The pull-back $\omega_{FS}^{[k]} = \psi_k^*\omega_{FS}^{loc}$ is a symplectic
structure on $U_k$ which is compatible with the complex structure $J_{U_k}$,
and one can check the following.

\begin{lemma}\label{LemmaLocFS}
$\omega_{FS}^{[k]} = i\del\delb\psi_k^*\rho$.
\end{lemma}

\begin{proof}
Recall that $d=\del + \delb$ and $d^c=-i(\del - \delb )$,
so that $dd^c = 2i\del\delb$. This gives
$i\del\delb\psi_k^*\rho = \frac{1}{2}dd^c\psi_k^*\rho$. One also has
$\psi_k^*d^c\rho = d^c\psi_k^*\rho$, because $d^c\rho = -d\rho \circ J_{\CC^d}$
and $d\psi_k$ is complex linear for any holomorphic chart $\psi_k$. In conclusion
$\frac{1}{2}dd^c\psi_k^*\rho = \psi_k^*\left(\frac{1}{2}dd^c\rho\right) = \psi_k^*\omega_{FS}^{loc}$.
\end{proof}

\begin{remark}
A popular notation in symplectic topology is $d^\CC f=df\circ J$. This is related to the
notation used in this article by $d^c = -d^\CC$.
\end{remark}

If $0\leq k,k'\leq d$, on $U_k\cap U_{k'}$ one can check that 
$\omega^{[k]}_{FS}=\omega^{[k']}_{FS}$,
and thus the local symplectic structures on each $U_k$ glue to a global
symplectic structure $\omega_{FS}$ on $\PP^d$ which is compatible with
the complex structure $J_{\PP^d}$, and is known as the Fubini-Study form.
However, since
$$\psi_k^*\rho = \log\left( \sum_{0\leq t\leq d}\lvert x_t\rvert^2\right)-2\log|x_k|$$
the local K\"{a}hler potentials $\psi_k^*\rho$ of Lemma \ref{LemmaLocFS}
do not match on the overlaps, and thus do not glue to a global potential for
$\omega_{FS}$. In what follows we will be mostly interested in the open set
$U=U_0\cap\ldots\cap U_d$, and we get rid of this ambiguity by adopting the
following convention.

\begin{definition}\label{DefFSPotential}
The function $\psi_0^*\rho_{|U}:U\to\RR$ is called the Fubini-Study potential
of the symplectic structure ${\omega_{FS}}_{|U}$.
\end{definition}

\subsection{A basis of characters}

Let $\Sigma$ a complete fan of cones in $\RR^n$ (Definition \ref{DefFan}), and assume
that $\phi:\RR^n\to\RR$ is a rational PL function on $\Sigma$ (Definition \ref{DefPLFunction})
that is strictly concave (Definition \ref{DefConcavity}). Since the $\QQ$-divisor $D_\phi$
is ample, for large $r\in\NN^+$ one has $rD_\phi=D_{r\phi}$ Cartier and very ample. Denoting
$d(r)=\operatorname{dim}\Gamma(X(\Sigma),\cO(D_{r\phi}))$, the basis of
characters associated to lattice points of the section polytope $P_{r\phi}\cap\ZZ^n=\{m_1,\ldots ,m_{d(r)}\}$
(Definition \ref{DefSectionPolytope}) defines a map

$$\nu_{r\phi}:(\CC^\times )^n\to \PP^{d(r)-1} \quad , \quad \nu_{rD_\Sigma}(z_1,\ldots ,z_n)=[\chi^{m_1}(z_1,\ldots ,z_n):\cdots :\chi^{m_{d(r)}}(z_1,\ldots ,z_n)]$$
which extends to a closed embedding of $X(\Sigma)$ in $\PP^{d(r)-1}$, and the Fubini-Study
symplectic structure on the target complex projective space induces a symplectic
structure $\omega_{r\phi}=\nu_{r\phi}^*\omega_{FS}$ which is compatible
with the complex structure $J_{(\CC^\times)^n}$ of the maximal torus orbit $X(\Sigma)$.

\subsection{The induced K\"{a}hler potential}

The Fubini-Study potential on the set $U\subset \PP^{d(r)-1}$ of Definition \ref{DefFSPotential}
induces a K\"{a}hler potential for the symplectic structure $\omega_{r\phi}$ on $(\CC^\times)^n$ as follows.

\begin{lemma}\label{LemmaInducedPotential}
$\omega_{r\phi} = i\del\delb \nu_{r\phi}^*\psi_0^*\rho$.
\end{lemma}

\begin{proof}
The map $\nu_{r\phi}$ is holomorphic with $\nu_{r\phi}((\CC^\times)^n)\subset U$,
and arguments analogous to Lemma \ref{LemmaLocFS} apply.
\end{proof}

More explicitly, the induced K\"{a}hler potential is
$$\nu_{r\phi}^*\psi_0^*\rho = \log\left( \sum_{1\leq t\leq d(r)}|\chi^{m_t}|^2 \right) - 2\log |\chi^{m_1}| \quad ;$$
up to summing to $\phi$ a linear function, one can assume that $0\in P_{r\phi}\cap \ZZ^n$ is one of the lattice points, and
taking this to be $m_1=0$ one gets $\log|\chi^{m_1}| = \log |\chi^0| = \log 1 = 0$.

\begin{definition}\label{DefInducedPotential}
The function $F_{r\phi}=\log \left( \sum_{m\in P_{r\phi}\cap \ZZ^n}|\chi^m|^2 \right)$
is called the K\"{a}hler potential of $\omega_{r\phi}$ on $(\CC^\times)^n$ induced by
the divisor $D_{r\phi}$.
\end{definition}

For future reference, denote $\theta_{r\phi} = \frac{1}{2}d^cF_{r\phi}$
the primitive of $\omega_{r\phi}$ induced by the K\"{a}hler potential $F_{r\phi}$.

\subsection{Formulas in complex coordinates}

We record here for later use some explicit formulas expressing the differential forms
$\theta_{r\phi}$ and $\omega_{r\phi}$ in terms of the coordinate $z\in(\CC^\times)^n$:
\vspace{2em}
$$\theta_{r\phi}= -\frac{i}{2}\sum_{k=1}^n\frac{\sum_{m\in P_{r\phi}\cap\ZZ^n}m^{(k)}|\chi^m|^2}{\sum_{m\in P_{r\phi}\cap\ZZ^n}|\chi^m|^2}
	\left(\frac{dz_k}{z_k}-\frac{d\overline{z}_k}{\overline{z}_k} \right)$$
\vspace{2em}
$$\omega_{r\phi}=i\sum_{1\leq h<k\leq n}
\frac{\sum_{m\in P_{r\phi}\cap\ZZ^n}m^{(h)}m^{(k)}|\chi^m|^2}{\sum_{m\in P_{r\phi}\cap\ZZ^n}|\chi^m|^2}
\left(\frac{dz_h}{z_h}\wedge\frac{d\overline{z}_k}{\overline{z}_k} - \frac{d\overline{z}_h}{\overline{z}_h}\wedge\frac{dz_k}{z_k} \right)$$
$$-i\sum_{1\leq h<k\leq n}
\frac{\left(\sum_{m\in P_{r\phi}\cap\ZZ^n}m^{(h)}|\chi^m|^2\right)\left(\sum_{m\in P_{r\phi}\cap\ZZ^n}m^{(k)}|\chi^m|^2\right)}{\left(\sum_{m\in P_{r\phi}\cap\ZZ^n}|\chi^m|^2\right)^2}
\left(\frac{dz_h}{z_h}\wedge\frac{d\overline{z}_k}{\overline{z}_k} - \frac{d\overline{z}_h}{\overline{z}_h}\wedge\frac{dz_k}{z_k} \right)$$
$$+i\sum_{k=1}^n\left(\frac{\sum_{m\in P_{r\phi}\cap\ZZ^n}(m^{(k)})^2|\chi^m|^2}{\sum_{m\in P_{r\phi}\cap\ZZ^n}|\chi^m|^2} - \frac{\left(\sum_{m\in P_{r\phi}\cap\ZZ^n}m^{(k)}|\chi^m|^2\right)^2}{\left(\sum_{m\in P_{r\phi}\cap\ZZ^n}|\chi^m|^2\right)^2} \right)\left(\frac{dz_k}{z_k}\wedge\frac{d\overline{z}_k}{\overline{z}_k} \right)$$
\vspace{2em}

\section{Polyhedral Hamiltonians}\label{SecPolyhedralHam}

\subsection{Radial coordinates and bump functions}

For each ray $\rho\in\Sigma(1)$ introduce a coordinate $r_\rho$ measuring the distance of a point
from the facet $\langle m,u_\rho\rangle = r\phi(u_\rho)$ of the section polytope $P_{r\phi}$,
with interior points having positive distance. As a result, one has $m\in P_{r\phi}$
if and only if $r_\rho(m) \geq 0$ for all rays $\rho\in \Sigma (1)$.

\begin{definition}\label{DefRadialCoordinate}
The radial coordinate $r_\rho : \RR^n \to \RR$ associated to $\rho\in\Sigma(1)$
is the function $r_\rho (m)=\langle m, u_\rho\rangle -r\phi(u_\rho)$. 
\end{definition}

For each $\epsilon_\rho \in (0,\infty)$ also introduce a function $q_{\epsilon_\rho}:[0,\infty )\to [0,1]$
given by
$$
q_{\epsilon_\rho}(x)=
\begin{cases}
\exp\left( -\frac{x^2}{\epsilon_\rho^2(\epsilon_\rho^2-x^2)} \right) \quad \textrm{for} \; x\in [0,\epsilon_\rho )\\
0 \quad \textrm{for} \; x\in [\epsilon_\rho,\infty )
\end{cases}
$$

\begin{definition}\label{DefBumpFunction}
The function $q_{\epsilon_\rho}$ is the bump function associated to $\rho\in\Sigma(1)$
with smoothing parameter $\epsilon_\rho$.
\end{definition}

We record below some basic properties of the bump function that will be useful later.

\begin{lemma}\label{LemmaBumpProperties}
The bump function $q_{\epsilon_\rho}$
satisfies the following properties:
\begin{enumerate}
	\item it is smooth, with $q_{\epsilon_\rho}(0)=1$ and $q_{\epsilon_\rho}(x)\geq 0$ with support $[0,\epsilon_\rho )$ ;
	\item $q_{\epsilon_\rho}'(x)\leq 0$ with equality if and only if $x=0$ or $x\in [\epsilon_\rho,\infty )$ ;
	\item there exists $x_{\epsilon_\rho}\in (0,\epsilon_\rho)$ such that $q_{\epsilon_\rho}''(x_{\epsilon_\rho})=0$,
		$q_{\epsilon_\rho}''(x) < 0$ for $x\in (0,x_{\epsilon_\rho})$ and $q_{\epsilon_\rho}''(x) > 0$ for $x\in (x_{\epsilon_\rho}, \epsilon_\rho)$ ;
	\item $\inf_{\epsilon_\rho > 0}q_{\epsilon_\rho}'(x_{\epsilon_\rho})=-\infty$ ;
	\item $\operatorname{inf}\{\; \epsilon_\rho\in (0,1) \; : \; q_{\epsilon_\rho}'(x_{\epsilon_\rho})\notin\QQ \; \}=0$ .
\end{enumerate}
\end{lemma}

\begin{proof}
(1) This holds by direct inspection of the definition of bump function.\\
(2) It follows from (1) after computing $q_{\epsilon_\rho}'(x)=-q_{\epsilon_\rho}(x)2x(\epsilon_\rho^2-x^2)^{-2}$.\\
(3) One computes $q_{\epsilon_\rho}''(x)=2q_{\epsilon_\rho}(x)g_{\epsilon_\rho}(x)(\epsilon_\rho^2-x^2)^{-4}$,
and verifies that for $\epsilon_\rho <1$ the function $g_{\epsilon_\rho}$ is 
differentiable and strictly increasing
on $[0,\epsilon_\rho]$, with $g_{\epsilon_\rho}(0)<0$ and $g_{\epsilon_\rho}(\epsilon_\rho)>0$.
The claim follows from (1) and the intermediate value theorem.\\
(4) By the mean value theorem, there exists some point $y_{\epsilon_\rho} \in (0,\epsilon_\rho)$
such that $q_{\epsilon_\rho}'(y_{\epsilon_\rho})=(q_{\epsilon_\rho}(\epsilon_\rho)-q_{\epsilon_\rho}(0))\epsilon_\rho^{-1}=-\epsilon_\rho^{-1}$.
By (3) $q_{\epsilon_\rho}'$ achieves its minimum at $x_{\epsilon_\rho}$, so that
$q_{\epsilon_\rho}'(x_{\epsilon_\rho})\leq q_{\epsilon_\rho}'(y_{\epsilon_\rho})=-\epsilon_\rho^{-1}$.\\
(5) It suffices to prove that $q'(x_{\epsilon_\rho})$ is a continuous function of
$\epsilon_\rho\in (0,1)$. One can compute the function $g_{\epsilon_\rho}$ mentioned in (3) explicitly, and it is
$g_{\epsilon_\rho}(x)=3x^4+(2-2\epsilon_\rho^2)x^2-\epsilon_\rho^4$.
This polynomial has a unique positive root, and it must be the point $x_\epsilon\in (0,\epsilon_\rho)$ of (3).
Since the polynomial is biquadratic, one can compute
$$x_{\epsilon_\rho}=\frac{2\epsilon_{\rho}^2-2+(16\epsilon_\rho^4-8\epsilon_\rho^2+4)^{1/2}}{6} \quad .$$
In particular, $x_{\epsilon_\rho}$ is a continuous function of $\epsilon_\rho$
and so is $q'(x_{\epsilon_\rho})$.
\end{proof}

\subsection{Smoothing the boundary}

Applying the bump functions to the radial coordinates introduced earlier one
gets smoothings of the polyhedral boundary $\partial P_{r\phi}$
that depend on the parameter $\epsilon = (\epsilon_\rho)_{\rho\in\Sigma (1)}\in \RR_{> 0}^{\Sigma (1)}$.
More precisely, consider the function
$$h_{r\phi}(m) = \sum_{\rho\in\Sigma (1)}q_{\epsilon_\rho}(r_\rho(m)) \quad ;$$
the smoothings will arise as level sets of this function.

\begin{lemma}\label{LemmaBumpAndPolytope}
The smoothing function $h_{r\phi,\epsilon}$ and the polytope $P_{r\phi}$
are related in the following way:
\begin{enumerate}
	\item $h_{r\phi,\epsilon}^{-1}(0)$ is a polytope contained in the
	interior of $P_{r\phi}$, and if $\epsilon_\rho \leq -r\phi(u_\rho)$ for all $\rho\in\Sigma(1)$
	it contains the origin ;
	\item if $\sigma\in \Sigma$ is any cone of positive dimension and $F_\sigma\subset \partial P_{r\phi}$
	is the corresponding face, then ${h_{r\phi,\epsilon}}_{|F_\sigma}\geq |\sigma(1)|$ .
\end{enumerate}
\end{lemma}

\begin{proof}
(1) By Lemma \ref{LemmaBumpProperties} $q_{\epsilon_\rho}\geq 0$ for all $\rho\in\Sigma(1)$
and $q_{\epsilon_\rho}(r_\rho(m))=0$ if and only if $r_\rho(m)\geq \epsilon_\rho$. It follows
that $h_{r\phi,\epsilon}(m)=0$ if and only if $r_\rho(m)\geq \epsilon_\rho$ for all $\rho\in\Sigma(1)$.
This condition is satisfied by $m=0$ when $\epsilon_\rho \leq -r\phi(u_\rho)$ for all $\rho\in\Sigma(1)$,
by Definition \ref{DefRadialCoordinate} of radial coordinate.
\\
(2) If $m\in F_\sigma$ then $r_\rho(m)=0$ and $q_{\epsilon_\rho}(r_\rho(m))=1$ for all $\rho \in \sigma(1)$.
Therefore
$$h_{r\phi,\epsilon}(m) = \sum_{\rho\in\Sigma (1)}q_{\epsilon_\rho}(r_\rho(m))
= \sum_{\rho\in\sigma (1)}q_{\epsilon_\rho}(r_\rho(m))+\sum_{\rho\notin\sigma (1)}q_{\epsilon_\rho}(r_\rho(m))
\geq |\sigma(1)| \quad .$$
\end{proof}

\begin{proposition}\label{PropLevelSpheres}
For all $\delta\in (0,1)$, the level set $h_{r\phi,\epsilon}^{-1}(\delta)$
is a smooth hypersurface contained in the interior of $P_{r\phi}$ and it is
homeomorphic to a sphere.
\end{proposition}

\begin{proof}
Denote $S^{n-1}$ and $D^n$ the sets of $m\in \RR^n$ with $|m|=1$ and $|m|\leq 1$ respectively.
For any direction $v\in S^{n-1}$, denote $\gamma_v(t)=tv$ for $t\in [0,\infty)$
the ray emanating from the origin (which is always contained in the interior of
$P_{r\phi}$) with direction $v$. Since $P_{r\phi}$ is convex, one has
$\gamma_v^{-1}(P_{r\phi})=[0,L_v]$ for some constant $L_v> 0$. We claim that for
any $\delta\in (0,1)$ and $v\in S^{n-1}$ the intersection $h_{r\phi,\epsilon}^{-1}(\delta)\cap \gamma_v([0,L_v])$
consists of a single point. To see this, consider the derivative
$$(h_{r\phi,\epsilon}(\gamma_v(t)))' = (\nabla h_{r\phi,\epsilon})(\gamma_v(t))\cdot v \quad ;$$
since the partial derivative of $h_{r\phi,\epsilon}$ with respect to the
$k$-th coordinate of $m=(m^{(1)},\ldots ,m^{(n)})$ is
$$\frac{\partial}{\partial m^{(k)}}h_{r\phi,\epsilon} = \sum_{\rho\in\Sigma(1)}
q'_{\epsilon_\rho}(r_\rho(m))\frac{\partial}{\partial m^{(k)}}(r_\rho(m))=
\sum_{\rho\in\Sigma(1)}q'_{\epsilon_\rho}(r_\rho(m))u_\rho^{(k)}$$
one gets
$$(h_{r\phi,\epsilon}(\gamma_v(t)))'=\sum_{k=1}^n\left(\sum_{\rho\in\Sigma(1)}q'_{\epsilon_\rho}(r_\rho(\gamma_v(t)))u_\rho^{(k)}\right)v^{(k)}
=\sum_{\rho\in\Sigma(1)}q'_{\epsilon_\rho}(r_\rho(\gamma_v(t)))\langle v, u_\rho\rangle \quad .$$
Now recall that $q'_{\epsilon_\rho}\leq 0$ by Lemma \ref{LemmaBumpProperties}, and observe
that $q'_{\epsilon_\rho}(r_\rho(\gamma_v(t)))\neq 0$ implies $\langle v,u_\rho \rangle < 0$.
Indeed in this case $0\leq r_{\epsilon_\rho}(\gamma_v(t)) < \epsilon_\rho$
and $\gamma_v(t)$ must form an obtuse angle with $u_\rho$, which is normal to the facet
$F_\rho \subset \partial P_{r\phi}$ and points inside the polytope. In conclusion, for
each $v\in S^{n-1}$ the function $h_{r\phi,\epsilon}(\gamma_v(t))$ is increasing on $[0,L_v]$
and strictly increasing away from $\gamma_v([0,L_v])\cap h_{r\phi,\epsilon}^{-1}(0)$;
by Lemma \ref{LemmaBumpAndPolytope} the latter set is a proper closed subinterval of $[0,L_v]$,
which contains the origin if $\epsilon_\rho\leq r$ for all $\rho\in\Sigma(1)$. From this
discussion it follows that the function $\psi(m)=m/|m|$ restricts
for each $\delta\in (0,1)$ to a continuous bijection $\psi_\delta = \psi_{|h_{r\phi,\epsilon}^{-1}(\delta)}$
between a compact set and the Hausdorff space $S^{n-1}$, and thus it is a homeomorphism.
The level set $h_{r\phi,\epsilon}^{-1}(\delta)$ is a smooth manifold by the implicit
function theorem, since the gradient of $h_{r\phi,\epsilon}$ never vanishes along it
by the calculation above.
\end{proof}

\begin{remark}
The nonempty level sets $h_{r\phi_\epsilon}^{-1}(\delta)$ with $\delta\geq	1$ are also
smooth manifolds by the arguments of Proposition \ref{PropLevelSpheres}, but they are not
necessarily compact anymore; see Figure \ref{FigLevelSets} for some examples.
\end{remark}

\begin{remark}\label{RmkExoticSpheres}
If $n\neq 5$ the $h$-cobordism theorem implies that each level set $h_{rD_\Sigma,\epsilon}^{-1}(\delta)$
is diffeomorphic to the standard sphere $S^{n-1}$. However, we do not know if the statement
holds for $n=5$. Note that if an exotic $S^4$ exists, then it has a smooth embedding
in $\RR^5$; see Colding-Minicozzi-Pedersen \cite{CMP}.
\end{remark}

\subsection{The Hamiltonians}

Following Fulton \cite[Section 4.2]{Fu}, define the moment map $\mu_{r\phi}:(\CC^\times)^n\to\RR^n$
of the toric variety $X(\Sigma)$ with polarization $D_{r\phi}$ as
$$\mu_{r\phi}=\frac{1}{\sum_{m\in P_{r\phi}\cap\ZZ^n}|\chi^m|^2}
\sum_{m\in P_{r\phi}\cap\ZZ^n}|\chi^m|^2m \quad .$$
The image of this map is the interior of $P_{r\phi}$, and the real torus
$(S^1)^n\subset (\CC^\times)^n$ acts freely and transitively on its fibers.
Endowing $(\CC^\times)^n$ with the symplectic structure $\omega_{r\phi}$ of Section \ref{SecPotential},
one can check that the action of the real torus is Hamiltonian, and that
$\mu_{r\phi}$ is a moment map in the sense of symplectic topology as follows.
The primitive generator $u_\rho$ of each ray $\rho\in\Sigma(1)$
defines a one-parameter subgroup
$$\lambda_{u_\rho}:\CC^\times\to(\CC^\times)^n \quad , \quad
\lambda_{u_\rho}(t)=(t^{u_\rho^{(1)}},\ldots ,t^{u_\rho^{(n)}}) \quad ;$$
denote $S_{u_\rho}=\lambda_{u_\rho}(S^1)$ the corresponding circle subgroup and
$$X_{u_\rho}(z)=\frac{d}{d\alpha}\vert_{\alpha=0}\lambda_{u_\rho}(e^{i\alpha})\cdot z \quad \textrm{for} \; z\in (\CC^\times)^n \quad .$$

\begin{definition}\label{DefInfinitesimalAction}
The vector field $X_{u_\rho}$ on $(\CC^\times)^n$ is called infinitesimal action the
circle subgroup $S_{u_\rho}$.
\end{definition}

\begin{lemma}\label{LemmaInfinitesimalAction}
Writing $u_\rho\in\ZZ^n$ as $\sum_{l=1}^nu_\rho^{(l)}e_l$
one has $X_{u_\rho}=\sum_{l=1}^nu_\rho^{(l)}X_{e_l}$, where
$$X_{e_l}(z)=\frac{ z_l+\overline{z}_l }{ 2 } i \left( \frac{ \partial }{ \partial z_l }-\frac{ \partial }{ \partial\overline{z}_l } \right)
-\frac{ z_l-\overline{z}_l }{ 2i } \left( \frac{ \partial }{ \partial z_l }+\frac{ \partial }{ \partial\overline{z}_l } \right) \quad .$$
\end{lemma}

\begin{proof}
For any $1\leq l\leq n$ write $z_l=x_l+iy_l$, so that
$$\lambda_{u_\rho}(e^{i\alpha})\cdot z = (e^{i\alpha u^{(1)}}x_1,\ldots ,e^{i\alpha u^{(1)}}x_n)
+i(e^{i\alpha u^{(1)}}y_1,\ldots ,e^{i\alpha u^{(1)}}y_n) \quad .$$
Differentiating at $\alpha=0$ and using $i\frac{\partial}{\partial x_l}=\frac{\partial}{\partial y_l}$ one
gets
$$X_{u_\rho}(z)=u^{(1)}\left(x_1\frac{\partial}{\partial y_1}-y_1\frac{\partial}{\partial x_1}\right)
+\cdots +u^{(n)}\left(x_n\frac{\partial}{\partial y_n}-y_n\frac{\partial}{\partial x_n}\right) \quad .$$
The claim follows from the fact that
$$x_l\frac{\partial}{\partial y_l}-y_l\frac{\partial}{\partial x_l}=
\frac{ z_l+\overline{z}_l }{ 2 } i \left( \frac{ \partial }{ \partial z_l }-\frac{ \partial }{ \partial\overline{z}_l } \right)
-\frac{ z_l-\overline{z}_l }{ 2i } \left( \frac{ \partial }{ \partial z_l }+\frac{ \partial }{ \partial\overline{z}_l } \right) \quad .$$
\end{proof}

\begin{proposition}\label{PropositionHamiltonianAction}
The following properties hold:
\begin{enumerate}
	\item $\omega_{r\phi}(X_{u_\rho},-)=-d\langle \mu_{r\phi},u_\rho\rangle$ for all $\rho\in\Sigma(1)$ ;
	\item $\theta_{r\phi}$ is invariant under the action of the real torus $(S^1)^n\subset (\CC^\times)^n$ .
\end{enumerate}
\end{proposition}

\begin{proof}
(1) By linearity and Lemma \ref{LemmaInfinitesimalAction} it suffices to check that $\omega_{r\phi}(X_{e_l},-)=-d\langle \mu_{r\phi},e_l\rangle$ for $1\leq l\leq n$.
Observe that $dz_h(X_{e_l})=i\delta_{hl}z_h$ and $d\overline{z}_h(X_{e_l})=-i\delta_{hl}\overline{z}_h$,
so that
$$\left( \frac{dz_h}{z_h}\wedge\frac{d\overline{z}_k}{\overline{z}_k} \right)(X_{e_l},-)=
i\left( \delta_{kl}\frac{z_h}{z_h} + \delta_{hl} \frac{\overline{z}_k}{\overline{z}_k}\right) \quad .$$
For $1\leq h,k\leq n$ define the function
$$c_{hk}=\frac{\sum_{m\in P_{r\phi}\cap\ZZ^n}m^{(h)}m^{(k)}|\chi^m|^2}{\sum_{m\in P_{r\phi}\cap\ZZ^n}|\chi^m|^2}
-\frac{\left( \sum_{m\in P_{r\phi}\cap\ZZ^n}m^{(h)}|\chi^m|^2 \right)\left( \sum_{m\in P_{r\phi}\cap\ZZ^n}m^{(k)}|\chi^m|^2 \right)}
{\left( \sum_{m\in P_{r\phi}\cap\ZZ^n}|\chi^m|^2 \right)^2} \quad .$$
Using the formula for $\omega_{r\phi}$ in complex coordinates from Section \ref{SecPotential}
one gets
$$\omega_{r\phi}(X_{e_l},-)=-\sum_{1\leq h<k\leq n}c_{hk}
\left( \delta_{kl}\left(\frac{dz_h}{z_h}+\frac{d\overline{z}_h}{\overline{z}_h}\right)
+ \delta_{hl}\left(\frac{dz_k}{z_k}+\frac{d\overline{z}_k}{\overline{z}_k}\right) \right)$$
$$-\sum_{k=1}^nc_{kk}\delta_{kl}\left( \frac{dz_k}{z_k}+\frac{d\overline{z}_k}{\overline{z}_k} \right) \quad .$$
To compute $d\langle \mu_{r\phi},e_l\rangle$ in complex coordinates, observe that
for any $1\leq k\leq n$ one has
$$\frac{\partial}{\partial z_k}|\chi^m|^2 = \frac{m^{(k)}}{z_k}|\chi^m|^2 \quad \textrm{and}
\quad \frac{\partial}{\partial \overline{z}_k}|\chi^m|^2 = \frac{m^{(k)}}{\overline{z}_k}|\chi^m|^2 \quad ,$$
so that using $d=\partial + \overline{\partial}$ one gets
$$d\langle \mu_{r\phi} , e_l\rangle = \sum_{k=1}^nc_{lk}\left( \frac{dz_k}{z_k} + \frac{d\overline{z}_k}{\overline{z}_k}\right) \quad .$$
The desired equality follows by comparing the formulas in complex coordinates just found.\\
(2) Recall that $\theta_{r\phi}=\frac{1}{2}d^cF_{r\phi}$, where $F_{r\phi}$ is
the K\"{a}hler potential induced by the divisor $rD_\phi$ from Definition \ref{DefInducedPotential}.
Denoting $\phi_t$ the diffeomorphism of $(\CC^\times)^n$ given by the action of $t\in (S^1)^n$,
one has
$$\phi_t^*d^cF_{r\phi}=-dF_{r\phi}\circ J_{(\CC^\times)^n}\circ d\phi_t =
-dF_{r\phi}\circ d\phi_t\circ J_{(\CC^\times)^n} = -d(F_{r\phi}\circ \phi_t)\circ J_{(\CC^\times)^n}
=d^cF_{r\phi} \quad ,$$
where we have used that the action of $(S^1)^n$ is holomorphic and that $F_{r\phi}$
is invariant under it, since the modulus of each character $|\chi^m|$ is.
\end{proof}

\begin{remark}
Since $d\langle \mu_{r\phi},u_\rho\rangle= d(r_\rho\circ \mu_{r\phi})$, it follows
from Proposition \ref{PropositionHamiltonianAction} that $X_{u_\rho}$ is the Hamiltonian
vector field associated to the moment map lift $R_\rho=\mu_{r\phi}^*r_\rho$ of the radial
coordinate of Definition \ref{DefRadialCoordinate}.
\end{remark}

One can use the moment map to lift the smoothing function constructed earlier to an
$(S^1)^n$-invariant function on $(\CC^\times)^n$.

\begin{definition}\label{DefPolyhedralHamiltonian}
If $\Sigma$ is a complete fan of cones in $\RR^n$ and $\phi:\RR^n\to\RR$ is a strictly
concave rational PL function on $\Sigma$, for any $r\in\NN^+$ such that $r\phi$ is an
integral PL function call polyhedral Hamiltonian the function $H_{r\phi,\epsilon} = h_{r\phi,\epsilon}\circ \mu_{r\phi}$ .
\end{definition}

Since the moment fibers are Lagrangian, the nonempty level sets $H_{r\phi}^{-1}(\delta)$
with $\delta\in (0,\infty)$ are fibered by Lagrangian tori, and when $\delta\in (0,1)$
they have topology $S^{n-1}\times (S^1)^n$ by Proposition \ref{PropLevelSpheres}.

\section{Lattice points and periodic orbits}\label{SecOrbits}

We wish to understand the dynamics of the vector field $X_{H_{r\phi,\epsilon}}$
defined by the equation $-dH_{r\phi,\epsilon} = \omega_{r\phi}(X_{H_{r\phi,\epsilon}},-)$,
in particular its periodic orbits.

\subsection{Collar at infinity}

For each cone $\sigma\in\Sigma$ and $\epsilon\in\RR_{>0}^{\Sigma(1)}$ define the set
$$U_\sigma^\epsilon = \{ \; z\in (\CC^\times)^n \; : \; 0<r_\rho(\mu_{r\phi}(z))<\epsilon_\rho
 \; \textrm{for all} \; \rho\in\sigma(1) \; \} \quad .$$
These open sets give rise to locally closed sets
$$S_\sigma^\epsilon = U_\sigma^\epsilon \cap \left( \bigcap\limits_{\rho\in\Sigma(1)\setminus\sigma(1)}(U_\rho^\epsilon)^c \right)$$
for any $\sigma\in\Sigma$. 
 
\begin{definition}\label{DefCollarOfInfinity}
The collection of sets $\{S_\sigma^\epsilon\}_{\sigma\in\Sigma}$ is called collar at infinity
of shape $\Sigma$ and size $\epsilon$, and each set is called a stratum.
\end{definition}

The strata form a partition of $(\CC^\times)^n$, and each of them is invariant under the
action of the real torus $(S^1)^n\subset(\CC^\times)^n$. There is only one closed stratum,
corresponding to the trivial cone $\sigma=\{0\}$; this is also the zero
set of the polyhedral Hamiltonian $S_{\{0\}}^\epsilon=H_{r\phi,\epsilon}^{-1}(0)$ and thus
it is a union of constant orbits of $X_{H_{r\phi,\epsilon}}$. The union of all other strata
is an open set, which can be thought of as a neighborhood of infinity. 
The image of this union under the moment map
$\mu_{r\phi}$ resembles a collar of the boundary of the section polytope
$\partial P_{r\phi}$, hence the name collar at infinity.

\begin{remark}
The strata do not meet the frontier condition, i.e. it is
not true that $\overline{S_\sigma^\epsilon}\cap S_{\sigma'}^\epsilon\neq\emptyset$
implies $S_{\sigma'}^\epsilon\subseteq S_{\sigma}^\epsilon$ in general.
\end{remark}

\subsection{Dynamics in a stratum}

The dynamics of $X_{H_{r\phi,\epsilon}}$ on strata $S_\sigma^\epsilon$ with $\sigma\neq \{0\}$
can be described as follows.

\begin{proposition}\label{PropSmoothingVectorField}
If $z\in S_\sigma^\epsilon$ then $X_{H_{r\phi,\epsilon}}(z)=\sum_{\rho\in\sigma(1)}q'_{\epsilon_\rho}(r_\rho(\mu_{r\phi}(z)))X_{u_\rho}(z)$ .
\end{proposition}

\begin{proof}
By Definition \ref{DefPolyhedralHamiltonian} of smoothing Hamiltonian
$$H_{r\phi,\epsilon}(z)=\sum_{\rho\in\Sigma(1)}q_{\epsilon_\rho}(r_\rho(\mu_{r\phi}(z)))=
\sum_{\rho\in\sigma(1)}q_{\epsilon_\rho}(r_\rho(\mu_{r\phi}(z))) \quad ,$$
where the second equality holds because $z\in S_\sigma^\epsilon$ implies
$r_\rho(\mu_{r\phi}(z))\geq \epsilon_\rho$ for all $\rho\in\Sigma(1)\setminus\sigma(1)$,
hence $q_{\epsilon_\rho}(r_\rho(\mu_{r\phi}(z)))=0$. Differentiating at $z\in S_\sigma^\epsilon$
one gets
$$d_zH_{r\phi,\epsilon}=\sum_{\rho\in\sigma(1)}q_{\epsilon_\rho}'(r_\rho(\mu_{r\phi}(z)))
(d_{\mu_{r\phi}(z)}r_\rho) \circ (d_z\mu_{r\phi})=$$
$$
\sum_{\rho\in\sigma(1)}q_{\epsilon_\rho}'(r_\rho(\mu_{r\phi}(z)))
\langle d_z\mu_{r\phi}, u_\rho\rangle = -
\sum_{\rho\in\sigma(1)}q_{\epsilon_\rho}'(r_\rho(\mu_{r\phi}(z)))
(\omega_{r\phi})_z(X_{u_\rho}(z),\cdot) \quad ,$$
where the last equality holds by Proposition \ref{PropositionHamiltonianAction}. The desired formula
follows from the non-degeneracy of the symplectic structure $\omega_{r\phi}$.
\end{proof}

Observe that the vector $X_{u_\rho}(z)$ is tangent to
$\mu_{r\phi}^{-1}(\mu_{r\phi}(z))$, because moment fibers are invariant
under the action of the real torus $(S^1)^n\subset (\CC^\times)^n$. It follows
from Proposition \ref{PropSmoothingVectorField} that $X_{H_{r\phi,\epsilon}}$ is
tangent to the the moment fibers, and in each of them its dynamics is that of a linear
flow on a torus, with slope depending on the particular fiber.

\subsection{Rational slopes}

For each cone $\sigma\in \Sigma$, consider the set
$$L_\sigma = \{ \; d\in\QQ_{>0}^{\sigma(1)} \; : \; \sum_{\rho\in\sigma(1)}d_\rho u_\rho\in\sigma\cap\ZZ^n \; \}$$
of positive rational tuples giving lattice combinations of the primitive generators
$u_\rho$ of rays $\rho\in\sigma(1)$. The set $L_\sigma$ is closed under the sum operation.

\begin{definition}\label{DefSlopeSemigroup}
Call $L_\sigma\subset \QQ_{>0}^{\sigma(1)}$ the slope semigroup of $\sigma\in\Sigma$.
\end{definition}

Observe that if $\sigma$ is a smooth cone, then each vector of $\sigma\cap\ZZ^n$ is a unique
integer combination of the vectors $u_\rho$ with $\rho\in\Sigma(1)$, and
thus $L_\sigma\subset \ZZ_{> 0}^{\sigma(1)}$ in that case. The map
$$c_\sigma : \QQ_{>0}^{\sigma(1)}\to\RR^n \quad , \quad c_\sigma(d)=\sum_{\rho\in\sigma(1)}d_\rho u_\rho$$
has image $c_\sigma(L_\sigma)= \operatorname{int}(\sigma)\cap\ZZ^n$, and
is an isomorphism of semigroups with its image when $\sigma$ is a simplicial cone. When
$\sigma$ is not simplicial, $L_\sigma$ can be larger because vectors of $\sigma\cap\ZZ^n$
could be written as a rational combination of the vectors $u_\rho$ with $\rho\in\sigma(1)$
in multiple ways, corresponding to elements of the fibers of $c_\sigma$.
For any $d\in L_\sigma$ introduce
$$B^\epsilon_\sigma(d) = \{ \; z\in \operatorname{int}(S_{\sigma}^\epsilon) \; : \; q'_{\epsilon_\rho}(r_\rho(\mu_{r\phi}(z)))=-d_\rho \; \textrm{for all} \; \rho\in\sigma(1) \; \} \quad .$$
Since the vector field $X_{H_{r\phi,\epsilon}}$ is tangent to the moment fibers,
its flow preserves $B^\epsilon_\sigma(d)$ and Proposition \ref{PropSmoothingVectorField} implies
that the orbits contained in it are periodic.

\begin{definition}\label{DefFamilyOfSlopeD}
Call $B^\epsilon_\sigma(d)$ a family of orbits of slope $d\in L_\Sigma$.
\end{definition}

\begin{remark}
There are many periodic orbits that are not in any family of type $B^\epsilon_\sigma(d)$. However,
the orbits of period one are all contained in these families; compare Remark \ref{RemarkWhyLatticePoints}.
\end{remark}

\begin{definition}\label{DefDynamicalSupport}
If $H_{r\phi,\epsilon}$ is the polyhedral Hamiltonian associated to the PL function
$r\phi$ on $\Sigma$ with smoothing parameter $\epsilon$, its dynamical support in $\sigma\in\Sigma$
is the set
$$\operatorname{DS}_\sigma(r\phi,\epsilon)=\{ \; u\in\operatorname{int}(\sigma)\cap\ZZ^n \; : \; \exists d\in L_\sigma \; \textrm{with} \; B^\epsilon_\sigma(d)\neq\emptyset \; \textrm{and} \; c_\sigma(d)=u \; \} \quad .$$
\end{definition}

The main result of this section is the following.

\begin{theorem}\label{ThmTopologyOfFamilies}
If $H_{r\phi,\epsilon}$ is the polyhedral Hamiltonian associated to the PL function
$r\phi$ on $\Sigma$ with smoothing parameter $\epsilon$, then for any $\sigma\in\Sigma$ the following facts hold:
\begin{enumerate}
	\item the dynamical support $\operatorname{DS}_\sigma(r\phi,\epsilon)$ is finite ;
	\item for any $d\in c_\sigma^{-1}(\operatorname{DS}_\sigma(r\phi,\epsilon))$ the family $B^\epsilon_\sigma(d)$ is a smooth manifold
		diffeomorphic to a disjoint union of thickened tori $\operatorname{int}(D^{n-\dim\sigma})\times (S^1)^n$ . 
\end{enumerate}
\end{theorem}

\begin{proof}
(1) Suppose $d\in L_\sigma$ is such that $B^\epsilon_\sigma(d)\neq\emptyset$. Picking $z\in B^\epsilon_\sigma(d)\subset S^\epsilon_\sigma$
one has $0<r_\rho(\mu_{r\phi}(z))<\epsilon_\rho$ and $q'(r_\rho(\mu_{r\phi}(z)))=-d_\rho$
for all $\rho\in\sigma(1)$. As explained in Lemma \ref{LemmaBumpProperties}, one has
$q'(r_\rho(\mu_{r\phi}(z)))\geq q'(x_{\epsilon_\rho})$ and so $d_\rho \leq -q'(x_{\epsilon_\rho})$
for all $\rho\in\sigma(1)$, so that $|c_\sigma(d)|\leq -\sum_{\rho\in\sigma(1)}q'(x_{\epsilon_\rho})|u_\rho|$.
Since this bound is independent of $d\in L_\sigma$, the dynamical support
$\operatorname{DS}_\sigma(r\phi,\epsilon)$ is bounded and discrete, hence finite.\\
(2) It suffices to prove that $\mu_{r\phi}(B^\epsilon_\sigma(d))$ is diffeomorphic to a
disjoint union of copies of $\operatorname{int}(D^{n-\dim\sigma})$. From
the properties of bump functions discussed in Lemma \ref{LemmaBumpProperties}, for any
$\rho\in\sigma (1)$ the equation $q'_{\epsilon_\rho}(x)=-d_\rho$ has two solutions
$x=a(d_\rho),b(d_\rho)$ and one can assume without loss
of generality that $0 < a(d_\rho) < x_{\epsilon_\rho} < b(d_\rho) < \epsilon_\rho$. Write
$B^\epsilon_\sigma(d) = \mu_{r\phi}(\operatorname{int}(S^\epsilon_\sigma))\cap K_\sigma(d)$ with
$$K_\sigma(d)=
\bigcap_{\rho\in\sigma(1)} ( \{ \; m\in\RR^n \; : \; r_\rho(m)=a(d_\rho) \; \} \cup
\{ \; m\in\RR^n \; : \; r_\rho(m)=b(d_\rho) \; \} ) \quad .$$
The equations $r_\rho(m)=a(d_\rho)$ and $r_\rho(m)=b(d_\rho)$ define two distinct and
parallel affine hyperplanes in $\RR^n$, with normal direction $u_\rho$, so $K_\sigma(d)$
is a disjoint union of $2^{|\sigma(1)|}$ affine subspaces of dimension $n-\dim \operatorname{span}(u_\rho : \rho\in\sigma(1) )=n-\dim \sigma$,
one for each choice function $\sigma(1)\to \{a(d_\rho),b(d_\rho)\}$. Since $\mu_{r\phi}(\operatorname{int}(S^\epsilon_\sigma))$
is an open convex set, each connected component of $\mu_{r\phi}(\operatorname{int}(S^\epsilon_\sigma))\cap K_\sigma(d)$
is diffeomorphic to the interior of a ball of dimension $n-\dim\sigma$.
\end{proof}

\subsection{Periods}

\begin{lemma}\label{LemmaLinearFlowTorus}
If $v=(a_1/b_1,\ldots ,a_n/b_n)\in\QQ^n$ with $v\neq 0$ and $\operatorname{gcd}(a_i,b_i)=1$ for all $1\leq i\leq n$,
then the periodic orbit $\gamma :\RR/T\ZZ\to (\RR/\ZZ)^n$ with $\gamma(0)=0$ and $\gamma'(t)=v$
has period $T=\frac{|\operatorname{lcm}(b_1,\ldots ,b_n)|}{|\operatorname{gcd}(a_1,\ldots ,a_n)|}$.
\end{lemma}

\begin{proof}
Write $\gamma(t)=x_1(t)e_1+\ldots +x_n(t)e_n$, where $x_i(0)=0$ and $x_i'(t)=a_i/b_i$ for
all $1\leq i\leq n$, and observe that $x_i(t)=ta_i/b_i$ for all
$1\leq i\leq n$. If $T\in (0,\infty)$ is such that $\gamma(t+T)=\gamma(t)$ in
$(\RR/\ZZ)^n$ for all $t\in\RR$, then $x_i(t+T)=x_i(t)+c_i$ for some constants $c_i\in\ZZ$
and thus $Ta_i/b_i=c_i$ for all $1\leq i\leq n$. Since $\operatorname{gcd}(a_i,b_i)=1$
by assumption, $T=t_ib_i/d_i$ for some $d_i,t_i\in\ZZ$ with $\operatorname{gcd}(t_i,d_i)=1$
and $d_i|a_i$ for all $1\leq i\leq n$.
The number $T=|\operatorname{lcm}(b_1,\ldots ,b_n)|/|\operatorname{gcd}(a_1,\ldots ,a_n)|$
is the minimum among all $T\in (0,\infty)$ satisfying these properties, and therefore it
is the period of $\gamma$.
\end{proof}

\begin{remark}\label{RemarkWhyLatticePoints}
It follows from the proof of Lemma \ref{LemmaLinearFlowTorus} that, although periodic
orbits with irrational $v\in\RR^n$ exist, they must have irrational period.
Moreover if an orbit has period one then $v\in\ZZ^n$, because $b_i | \operatorname{lcm}
(b_1,\ldots ,b_n)=\pm\operatorname{gcd}(a_1,\ldots ,a_n)|a_i$ and $\operatorname{gcd}(a_i,b_i)=1$
for all $1\leq i\leq n$.
\end{remark}

\begin{proposition}\label{PropPeriods}
If $\gamma$ is a periodic orbit of $X_{H_{\epsilon,r\phi}}$ in the family $B^\epsilon_\sigma(d)$,
then its period is
$$T(\gamma)=\frac{1}{|\operatorname{gcd}(\langle c_\sigma(d),e_k\rangle : k=1,\ldots ,n)|} \quad .$$
In particular, $\gamma$ has period one if and only if $c_\sigma(d)$ is primitive.
\end{proposition}

\begin{proof}
Since $\gamma$ is in $B^\epsilon_\sigma(d)$, one has
$$X_{H_{\epsilon,r\phi}}(\gamma(t))=-\sum_{\rho\in\sigma(1)}d_\rho X_{u_\rho}(\gamma(t)) \quad .$$
By Lemma \ref{LemmaInfinitesimalAction} $X_{u_\rho}=\sum_{k=1}^nu_\rho^{(k)}X_{e_k}$, so that
$$X_{H_{\epsilon,r\phi}}(\gamma(t))=\sum_{k=1}^n\left( -\sum_{\rho\in\sigma(1)}d_\rho u_\rho^{(k)}
\right)X_{e_k}(\gamma(t))=\sum_{k=1}^n(-\langle c_\sigma(d),e_k\rangle )X_{e_k}(\gamma(t)) \quad .$$
The conclusion follows from Lemma \ref{LemmaLinearFlowTorus} and the fact that $c_\sigma(d)\in\ZZ^n$
since $d\in L_\sigma$.
\end{proof}

\section{Level sets are of contact type}\label{SecContactType}

\subsection{The candidate contact form}

For any $\epsilon\in\RR_{>0}^{\Sigma(1)}$ and $\delta\in (0,\infty)$ set
$$W_{\epsilon,\delta}(r\phi) = \{ \; z\in (\CC^\times)^n \; : \;
H_{\epsilon, r\phi}(z)\leq \delta \; \} \quad ;$$
this is a submanifold with boundary of $(\CC^\times)^n$ with a Lagrangian torus fibration
given by the moment map $\mu_{r\phi}$, and $\partial W_{\epsilon,\delta}(r\phi)=
H_{\epsilon, r\phi}^{-1}(\delta)$ is homeomorphic to $S^{n-1}\times(S^1)^n$ for $\delta\in (0,1)$
thanks to Proposition \ref{PropLevelSpheres} (however, it can be non-compact for $\delta\geq 1$).
Recall from Definition \ref{DefInducedPotential} that $(\CC^\times)^n$ has a K\"{a}hler
potential $F_{r\phi}$ induced by the divisor $rD_\phi$. The one-form $\theta_{r\phi}
=\frac{1}{2}d^cF_{r\phi}$ may or may not restrict to a contact form on $\partial W_{\epsilon,\delta}(r\phi)$,
depending on whether the dual vector field $X_{\theta_{r\phi}}$ points out along the
boundary $\partial W_{\epsilon,\delta}(r\phi)$ or not. This is equivalent to asking
whether $dH_{\epsilon ,r\phi}(X_{\theta_{r\phi}})>0$ along $H^{-1}_{\epsilon,r\phi}(\delta)$
or not. Below is a sufficient criterion for this to hold.

\begin{lemma}\label{LemmaCriterionForTransversality}
If $\theta_{r\phi}(X_{u_\rho})_{|U^\epsilon_\rho}<0$ for all $\rho\in\Sigma(1)$ then
$dH_{\epsilon ,r\phi}(X_{\theta_{r\phi}})>0$ along $H^{-1}_{\epsilon,r\phi}(\delta)$.
\end{lemma}

\begin{proof}
Observe that
$$dH_{\epsilon ,r\phi}(X_{\theta_{r\phi}})= -
\omega_{r\phi}(X_{H_{\epsilon,r\phi}},X_{\theta_{r\phi}})=
\theta_{r\phi}(X_{H_{\epsilon,r\phi}}) \quad .$$
Using the formula for the vector field associated to the smoothing Hamiltonian found in
Proposition \ref{PropSmoothingVectorField}, one gets
$$\theta_{r\phi}(X_{H_{\epsilon,r\phi}})=
\sum_{\rho\in\Sigma(1)}(q'_{\epsilon_\rho}(r_\rho\circ\mu_{r\phi}))\theta_{r\phi}(X_{u_\rho}) \quad .$$
Now $q'_{\epsilon_\rho}\leq 0$ by construction of the bump functions, and if $z\in(\CC^\times)^n$
has $q'_{\epsilon_\rho}(r_\rho(\mu_{r\phi}(z)))\neq 0$ then $z\in U^\epsilon_\rho$
and so $\theta_{r\phi}(X_{u_\rho})(z)<0$ by assumption; this implies that
$dH_{\epsilon ,r\phi}(X_{\theta_{r\phi}})> 0$ along $H^{-1}_{\epsilon,r\phi}(\delta)$,
because for $\delta \in (0,\infty)$ the level set $H_{\epsilon,r\phi}^{-1}(\delta)$ is contained
in the union of all $U^\epsilon_\rho$ with $\rho\in\Sigma(1)$, so $H_{\epsilon,r\phi}(z)=\delta$
implies $q'_{\epsilon_\rho}(r_\rho(\mu_{r\phi}(z)))\neq 0$ for some $\rho\in\Sigma(1)$.
\end{proof}

The quantity $\theta_{r\phi}(X_{u_\rho})$ in the criterion above is
controlled by the following distortion formula.

\begin{lemma}\label{LemmaDistortionFormula} (Distortion formula)
There exists $C_\rho\in\RR$ such that $\theta_{r\phi}(X_{u_\rho})=\langle \mu_{r\phi},u_\rho\rangle + C_\rho$
holds on $(\CC^\times)^n$.
\end{lemma}

\begin{proof}
By Proposition \ref{PropositionHamiltonianAction} the moment map $\mu_{r\phi}:(\CC^\times)^n\to\RR^n$
satisfies $\iota_{X_{u_\rho}}\omega_{r\phi}=-d\langle \mu_{r\phi},u_\rho\rangle$
for all $\rho\in\Sigma(1)$. Using $\omega_{r\phi}=d\theta_{r\phi}$ and Cartan's formula
$$\mathcal{L}_{X_{u_\rho}}\theta_{r\phi}=d(\iota_{X_{u_\rho}}\theta_{r\phi}-\langle \mu_{r\phi},u_\rho\rangle ) \quad ,$$
so the existence of $C_\rho$ follows from the fact that $(\CC^\times)^n$ is connected and
$\mathcal{L}_{X_{u_\rho}}\theta_{r\phi}=0$ because the action of the real torus $(S^1)^n\subset (\CC^\times)^n$
preserves $\theta_{r\phi}$, as explained in Proposition \ref{PropositionHamiltonianAction}.
\end{proof}

\subsection{Wrapping and averaging}

\begin{definition}\label{DefWrappingFunction}
For any $\rho\in\Sigma(1)$, call wrapping function of the one-parameter subgroup
$\lambda_{u_\rho}:\CC^\times\to (\CC^\times)^n$ the function
$$w_{\rho,r\phi}: (0,\infty)\to\RR \quad , \quad w_{\rho,r\phi}(a)= \int_{|t|=a}\lambda_{u_\rho}^*\theta_{r\phi} \quad .$$
\end{definition}

Thinking of $\lambda_{u_\rho}(\CC^\times)\subset (\CC^\times)^n$ as a holomorphic cylinder,
the wrapping function $w_\rho(a)$ encodes the periods of $\theta_{r\phi}$ along a family
of circles that cover it, one for each radius $a\in (0,\infty)$. These periods can also
be computed from the combinatorial data of the section polytope $P_{r\phi}$.

\begin{definition}\label{DefLatticeAverageFunction}
For any $\rho\in\Sigma(1)$, call lattice average function of $\rho$ the function
$$\overline{m}_{\rho,r\phi}:(0,\infty)\to\RR^n \quad , \quad \overline{m}_{\rho,r\phi}(a)=
\frac{\sum_{m\in P_{r\phi}\cap\ZZ^n}a^{2\langle m,u_\rho\rangle}m}{\sum_{m\in P_{r\phi}\cap\ZZ^n}a^{2\langle m,u_\rho\rangle}} \quad .$$
\end{definition}

At $a=1$ the average lattice functions of all rays $\rho\in\Sigma(1)$ agree and give:
$$\overline{m}_{\rho,r\phi}(1)=\frac{\sum_{m\in P_{r\phi}\cap\ZZ^n}m}{|P_{r\phi}\cap\ZZ^n|} \quad .$$

\begin{definition}\label{DefAverageLatticePoint}
Call $\overline{m}_{r\phi}=\overline{m}_{\rho,r\phi}(1)$ for any $\rho\in\Sigma(1)$
the average lattice point of the section polytope $P_{r\phi}$.
\end{definition}

The wrapping functions are related to the lattice average functions by the following.

\begin{proposition}\label{PropWrappingAverageFormula} (Wrapping-averaging formula)
For any $\rho\in\Sigma(1)$ one has $w_{\rho,r\phi}(a)=2\pi\langle \overline{m}_{\rho,r\phi}(a),u_\rho\rangle$ .
\end{proposition}

\begin{proof}
Writing $z=(z_1,\ldots ,z_n)\in (\CC^\times)^n$ and
$m=(m^{(1)},\ldots ,m^{(n)})\in P_{r\phi}\cap\ZZ^n$ one has
$$|\chi^m(z_1,\ldots ,z_n)|^2=|z_1|^{2m^{(1)}}\cdots |z_n|^{2m^{(n)}} \quad .$$
It follows that for any $\rho\in\Sigma(1)$ and $t\in\CC^\times$ one has
$|\chi^m(\lambda_{u_\rho}(t))|^2=|t|^{2\langle m,u_\rho\rangle}$. Using the explicit
formula for $\theta_{r\phi}$ in complex coordinates found in Section \ref{SecPotential},
the substitution $z_k=t^{u_\rho^{(k)}}$ for $1\leq k\leq n$ gives
$$(\lambda_{u_\rho}^*\theta_{r\phi})(t)=-\frac{i}{2}\sum_{k=1}^n
\frac{\sum_{m\in P_{r\phi}\cap\ZZ^n} m^{(k)}|t|^{2\langle m,u_\rho\rangle}}{\sum_{m\in P_{r\phi}\cap\ZZ^n} |t|^{2\langle m,u_\rho\rangle}}
u_\rho^{(k)}\left(\frac{dt}{t}-\frac{d\overline{t}}{\overline{t}} \right) \quad .$$
The claim follows from writing $t=ae^{i\alpha}$ for $a\in (0,\infty)$ and $\alpha\in [0,2\pi)$
and using Definition \ref{DefWrappingFunction} of wrapping function $w_\rho(a)$.
\end{proof}

There is a direct relation between the infinitesimal wrapping as $a\to 0$ and the
coefficients of the divisor $rD_\phi$ as combination of the prime torus invariant
divisors $D_\rho$ associated to rays $\rho\in\Sigma(1)$.

\begin{corollary}\label{CorInfinitesimalWrapping}
For any $\rho\in\Sigma$ one has $\lim_{a\to 0} w_{\rho, r\phi}(a)=2\pi\phi_{r\phi}(u_\rho)$.
\end{corollary}

\begin{proof}
Fix some ray $\rho\in\Sigma(1)$. Since the fan $\Sigma$ is complete, there exists a
top-dimensional cone $\sigma\in\Sigma(n)$ such that $\rho\in\sigma(1)$, and on $\sigma$
one has $(\phi_{r\phi})_{|\sigma}=\langle m_\sigma , \cdot \rangle$ for some lattice
point $m_\sigma\in\ZZ^n$. Since $u_\rho\in\rho\subseteq\sigma$, in particular
$\phi_{r\phi}(u_\rho)=\langle m_\sigma, u_\rho\rangle$ . From Proposition \ref{PropWrappingAverageFormula}
$w_\rho(a)=2\pi\langle \overline{m}_{\rho,r\phi}(a),u_\rho\rangle$, so the claim is
equivalent to proving that $\langle \overline{m}_{\rho,r\phi}(a),u_\rho\rangle\to \langle m_\sigma, u_\rho\rangle$
as $a\to 0$. By Definition \ref{DefLatticeAverageFunction} of lattice average function of $\rho$,
the quantity $\langle \overline{m}_{\rho,r\phi}(a),u_\rho\rangle$ is a rational function
of $a$, so the limit as $a\to 0$ is $\min_{m\in P_{r\phi}\cap\ZZ^n}\langle m,u_\rho\rangle$.
Since $u_\rho$ is normal to the facet $F_\rho\subset\partial P_{r\phi}$ and points inside
the polytope, this minimum is achieved precisely by the lattice points $m\in F_\rho\cap\ZZ^n$
of this facet. The claim follows from the observation that $m_\sigma\in F_\rho$, because
$\Sigma$ is the normal fan of $P_{r\phi}$ and $m_\sigma$ is the vertex dual to
the top-dimensional cone $\sigma$, while $\rho\subseteq \sigma$ is the ray dual
to the facet $F_\rho$.
\end{proof}

One can give a more explicit description of the constant $C_\rho\in\RR$ appearing
in the distortion formula of Lemma \ref{LemmaDistortionFormula}.

\begin{lemma}\label{LemmaDistortionConstant}
One has $C_\rho = 0$.
\end{lemma}

\begin{proof}
Recall the distortion formula in Lemma \ref{LemmaDistortionFormula}:
$\theta_{r\phi}(X_{u_\rho})=\langle \mu_{r\phi},u_\rho\rangle + C_\rho$.
Pulling back both sides along the one-parameter subgroup $\lambda_{u_\rho}:\CC^\times\to (\CC^\times)^n$
and integrating over the unit circle $S^1\subset\CC^\times$ one gets
$$w_\rho(1) = \int_{0}^{2\pi}\langle \mu_{r\phi}(\lambda_{u_\rho}(e^{i\alpha})),u_\rho\rangle d\alpha + 2\pi C_\rho \quad .$$
Since $\chi^m$ is a character one has $|\chi^m(\lambda_{u_\rho}(e^{i\alpha}))|=1$, and
by how the moment map is defined
$\langle \mu_{r\phi}(\lambda_{u_\rho}(e^{i\alpha})),u_\rho\rangle = \langle \overline{m}_{r\phi},u_\rho\rangle$.
The claim then follows from the fact that $w_\rho(1)=2\pi\langle \overline{m}_{r\phi},u_\rho\rangle$
by Proposition \ref{PropWrappingAverageFormula} and Definition \ref{DefAverageLatticePoint} of average lattice point.
\end{proof}

\begin{corollary}\label{CorExactTorus}
The Lagrangian torus $\mu_{r\phi}^{-1}(0)\subset (\CC^\times)^n$ is exact with respect
to $\theta_{r\phi}$.
\end{corollary}

\begin{proof}
Combining Lemma \ref{LemmaDistortionFormula} and Lemma \ref{LemmaDistortionConstant}, one
sees that $\theta_{r\phi}$ restricts to the zero one-form on $\mu_{r\phi}^{-1}(0)$.
\end{proof}

\subsection{Transversality of the Liouville vector field}

We prove below that positive level sets of polyhedral Hamiltonians as in Definition
\ref{DefPolyhedralHamiltonian} are hypersurfaces of contact type, as long as $\phi$
is negative on primitive generators of the rays of $\Sigma$.

\begin{theorem}\label{ThmContactType}
If $\phi(u_\rho)<0$ for all $\rho\in\Sigma(1)$, then $\theta_{r\phi}$ restricts to a contact form on the hypersurface
$\partial W_{\epsilon,\delta}(r\phi)=H_{\epsilon, r\phi}^{-1}(\delta)\subset (\CC^\times)^d$
for all $\delta\in (0,\infty)$ and $\epsilon\in \RR_{>0}^{\Sigma(1)}$ such that $\epsilon_\rho < -r\phi(u_\rho)$ for all $\rho\in\Sigma(1)$ .
\end{theorem}

\begin{proof}
Thanks to Lemma \ref{LemmaCriterionForTransversality}, it suffices to prove that
$\theta_{r\phi}(X_{u_\rho})_{|U^\epsilon_\rho}<0$ for all $\rho\in\Sigma(1)$. By the
distortion formula of Lemma \ref{LemmaDistortionFormula} and the calculation of the distortion
constant of Lemma \ref{LemmaDistortionConstant}, the latter condition is equivalent to
$$\langle \mu_{r\phi}(z),u_\rho\rangle < 0 \quad \textrm{for all} \quad z\in U_\rho^\epsilon \quad .$$
When $z\in U_\rho^\epsilon$ one
has $\langle \mu_{r\phi}(z),u_\rho\rangle < r\phi(u_\rho) +\epsilon_\rho$.
To achieve the desired inequality it suffices to choose
$\epsilon_\rho < -r\phi(u_\rho)$ for all $\rho\in\Sigma(1)$. This constraint can be met by
$\epsilon_\rho \in \RR_{>0}$ thanks to the assumption that $\phi(u_\rho)<0$ for all $\rho\in\Sigma(1)$.
\end{proof}

\bibliographystyle{abbrv}
\bibliography{biblio}

\end{document}